\documentclass[12pt]{amsart}
\usepackage{geometry} % see geometry.pdf on how to lay out the page. There's lots.
\usepackage{pgfplots}
\usepackage{psfrag}
\usepackage{amsmath}
\usepackage{calc}
\usepackage{systeme}
\usepackage{amsfonts}
\usepackage{amsthm}
\usepackage {amssymb}
\usepackage{algorithm2e}
\usepackage[applemac]{inputenc}
\usepackage{bbold}
\usepackage[numbers]{natbib}

\newtheorem{hyp}{Hypothesis}

\newtheorem{Theorem}{Theorem}[part]
\newtheorem{Lemma}{Lemma}[part]
\newtheorem{Rem}{Remark}
\newtheorem{Def}{Definition}
\newtheorem{Prop}{Proposition}[part]

\usepackage[foot]{amsaddr}

\title{Fokker-Planck equations of jumping particles and mean field games of impulse control}
\author{Charles Bertucci}

\address{Université Paris-Dauphine, PSL Research University,UMR 7534, CEREMADE, 75016 Paris, France}
\date{} % delete this line to display the current date

\begin{document}
\begin{abstract}
This paper is interested in the description of the density of particles evolving according to some optimal policy of an impulse control problem. We first fix sets on which the particles jump and explain how we can characterize such a density. We then investigate the coupled case in which the underlying impulse control problem depends on the density we are looking for : the mean field game of impulse control. In both cases, we give a variational characterization of the densities of jumping particles.
\end{abstract}
\maketitle
\tableofcontents
\section{Introduction}
\subsection{General introduction}
This paper is the second of a series devoted to the systematic study of mean field games (MFG for short) with optimal stopping or impulse control. In \citep{bertucci2017optimal} we developed an obstacle problem approach to solve a forward-backward system which models MFG with optimal stopping (without common noise). We here develop the same point of view for MFG with impulse control. As the definition of a Fokker-Planck equation associated with a density of players playing an impulse control problem is a difficult question in itself, it is the subject of the first part of this article. This part is independent from the MFG theory. The case of the master equation (i.e. when there is a common noise) will be treated in a subsequent work.

MFG model situations in which a continuum of indistinguishable players are playing a differential game. The evolution of the density of players is induced by the optimal choices the players make. In many situations, the costs involved in the game depend only on the density of players. Denoting by $u$ the value function of a generic player and by $m$ the density of players, a classical forward-backward MFG system during the time interval $(0,T)$ is

\[
\begin{cases}
-\partial_t u - \nu \Delta u + H(x,\nabla u) = f(m) ;\\
\partial_t m -  \nu \Delta m - div(D_p H(x, \nabla u) m) = 0 ;\\
m(0) = m_0 ; u(T) = g(m(T)).
\end{cases}
\]
where $H(x,p)$ is the hamiltonian of a continuous optimal control problem, $m_0$ is the initial condition for the density of players, $g$ and $f$ are respectively the terminal and running costs and $\nu > 0$ characterizes the intensity of the individual noises. A solution $(u,m)$ of this system corresponds to a Nash equilibrium for the game with an infinite number of players. This system, as well as MFG, have been introduced in \citep{lasry2006jeux, lasry2006jeux2,lasry2007mean} by J.-M. Lasry and P.-L. Lions. In these papers they proved general conditions under which the existence and uniqueness hold for this problem. Such games have been introduced independently in \citep{huang2006large}. We also refer to \citep{lions2007cours, cardaliaguet2010notes} for more results on this system. MFG models have a wide range of applications, see \citep{achdou2016long, gueant2009reference, gueant2011mean} for examples. Many interesting questions have also been raised around this system, we can cite for example long time average \citep{cardaliaguet2017long}, learning \citep{cardaliaguet2017learning}, the difficult problem of the convergence of the system of $N$ players as $N$ goes to infinity and the presence of a common noise \citep{cardaliaguet2015master}. Numerical methods are also being developed, let us cite \citep{achdou2010mean, bricen?o2018proximal} for instance. Let us also mention that a powerful probabilistic point of view on MFG has been developed, we refer to \citep{carmona2017probabilistic, carmona2013probabilistic, lacker2015mean} for more details on this point of view. In this paper, we generalize the results of the existence and uniqueness of the previous system to the case in which the players face an impulse control problem. Concerning closely related works, several optimal control problems, related to the impulse control problem, have been studied in a MFG setting. Optimal stopping or obstacle problems have been studied in \citep{bertucci2017optimal, gomes2014obstacle, nutz2016mean, carmona2017mean}, singular controls in \citep{fu2017mean, guo2017mean} and optimal switching in \citep{gomes2016weakly}.

\subsection{Impulse control problems}
Impulse control problems have been studied since the 70s. We refer to the work of A. Bensoussan and J.-L. Lions (see \citep{bensoussan1984impulse}) for a complete presentation of the problem. The terminology impulse control refers to an optimization problem in which the state is driven by a stochastic ordinary differential equation with jumps :

\begin{equation}\label{jumpprocess}
\begin{cases}
\forall t \in (\tau_i, \tau_{i+1}), dX_t = \sqrt{2 \nu}dW_t; \\
X_{\tau_i^+} = X_{\tau_i^-} + \xi_i ;
\end{cases}
\end{equation}
where $(W_t)_{t \geq 0}$ is a brownian motion under a standard probability space $(\Omega, \mathcal{A}, \mathbb{P})$. The jump $\xi_i$ occurs at time $\tau_i$ and is controlled by the player. The jumps are characterized by the random sequence of stopping times $(\tau_i)_{i \geq 0}$ and the random sequence of jumps $(\xi_i)_{i \geq 0}$. Those two sequences are chosen the controls and are adapted to the brownian motion $(W_t)_{t \geq 0}$ in the sense that the sequence $(\tau_i)_{i \geq 0}$ is indeed a sequence of stopping times for this brownian motion and that $(\xi_i, \tau_i)$ is measurable with respect to the $\sigma$-algebra generated by $(W_t)_{0 \leq t \leq \tau_i}$. We assume that $(\xi_i)_{i \geq 0}$ is valued in a set $K \subset \mathbb{T}^d$. Denoting by $f$ the running cost and by $k(x, \xi)$ the cost paid to use the jump $\xi$ while on the position $x$, we define the value function $u$ by
\[
u(t,x) := \inf_{(\tau_i)_i, (\xi_i)_i} \mathbb{E}[ \int_t^T f(s,X_s) ds + \sum_{ i =1 }^{\#(\tau_j)_j} k(X_{\tau_i^-}, \xi_i)];
\]
where the infimum is taken over all sequences $((\tau_i)_i , (\xi_i)_i)$ which are adapted to the brownian motion in the sense prescribed above and which satisfy the fact that $(\tau_i)_i$ is an increasing sequence and that $(\xi_i)_i$ is valued in $K$. The trajectory $(X_s)_{s\geq t}$ is given by (\ref{jumpprocess}) with the initial condition $X_t = x$. Under some assumptions on the costs( which are detailed in the appendix), the value function $u$ satisfies, in the sense of quasi variational inequality (QVI) (which is also detailed in the appendix), the Hamilton-Jacobi-Bellman equation :
\[
\begin{cases}
\max(-\partial_t u - \nu \Delta u -f, u(t,x) - \inf_{\xi \in K} \{k(x,\xi) + u(t,x + \xi)\}) = 0.; \\
u(T) = 0.
\end{cases}
\]
\subsection{The density of players} In a MFG context, the main question we are addressing is how will evolve an initial density of players, if those players are facing the same impulse control problem. Intuitively, the density of players $m$ has to satisfy (formally) at least some requirements :
\begin{itemize}
\item $\partial_t m -  \nu \Delta m = 0$ where it is optimal for the players to wait and not to jump and where no player is arriving.
\item $m = 0$ where it is strictly suboptimal not to jump.
\item The flux of arriving players at $x$ is equal to the sum over $\xi$ of the flows of players which choose to use the jump $\xi$ at $x-\xi$.
\end{itemize}
Let us note that, at least formally, we talk about parts of the space on which it is optimal to jump (i.e. to use a control to make the process $(X_t)_t$ jumps) because all the players being indistinguishable, if it is optimal for one player to jump, then it is optimal for all the players to jump.\\
\\
The problem of finding a density $m$ which satisfies the above requirements is not classical, mostly because there is no particular assumption on how the players use their controls.  We focus on the problem of modeling the evolution of a density of jumping players in the first part of this article. We build a dual characterization of the solution of the Fokker-Planck "equation". We fix a function $V(t,x,\xi)$ which describes wether or not the players use the jump $\xi$ at the position $x$ and time $t$. Then we construct a density of players $m$ which satisfies the required properties and thus solves a Fokker-Planck equation of jumping particles. The characterization of this density relies on the fact that we can interpret such Fokker-Planck equations as dual equations of QVI. The construction of such a solution uses a penalized version of the problem in which we can write properly the PDE satisfied by the density $m$. We then find a priori estimates which allow us to pass to the limit in this penalized version of the problem, while obtaining a characterization of the limit density. We also give results concerning the stationary case.\\
\\
In the second part of this article, we present results on the uniqueness and existence for the impulse control problem in MFG. We also recall that in view of the results of \citep{bertucci2017optimal} we expect the solutions of the MFG system to be mixed solutions, meaning that optimal strategies are random in general as the Nash equilibria of the MFG can be mixed equilibria. We end this second part by giving results on the stationary case and on the optimal control interpretation of such MFG.\\
\\
We present in two appendixes new results concerning QVI and the heat equation in a time dependent domain, which we need in this article and which are somehow independent of it.

\part{Fokker-Planck equation of jumping particles}
In this part we present the variational formulation of the Fokker-Planck equation satisfied by the density of jumping particles. We work in the $d$ dimensional torus (denoted by $\mathbb{T}^d$) in a time dependent setting, except for the last section which is concerned with the stationary case. The positive real number $T$ is the final time and $m_0 \in L^2(\mathbb{T^d})$ is the initial density at time $t = 0$. The aim of this part is to construct a suitable notion to characterize densities of jumping particles. By opposition to Fokker-Planck equations of jumping processes, we do not want to model populations of particles which are driven by Poisson processes or other jump processes of the sort. Namely if a unique jump $\xi$ is possible, we are interested in building solutions for $\epsilon > 0$ of
\begin{equation}\label{onejump}
\begin{cases}
\partial_t m_{\epsilon} -  \nu\Delta m_{\epsilon} + \frac{1}{\epsilon} \mathbb{1}_{A}m_{\epsilon} - \frac{1}{\epsilon} (\mathbb{1}_{A}m_{\epsilon})(t,x- \xi) = 0; \\
m(0) = m_0;
\end{cases}
\end{equation}
and passing to the limit $\epsilon \to 0$. The interpretation of (\ref{onejump}) is that the particles jump $\xi$ further if they are in the set $A$ with a probability given by an exponential law of parameter $\epsilon^{-1}$. The interpretation of the limit $m$ of solutions of (\ref{onejump}) is that it describes particles evolving only along brownian trajectories in $A^c$ and which jump $\xi$ further once they reach $A$. If $A$ is a smooth closed set such that the reaching time of $\partial A$ is well defined, then the trajectory $(X_s)_{s \geq 0}$ of a generic particle is defined by
\[
\begin{cases}
dX_s = \sqrt{2 \nu} dW_s,\forall i,  \forall s \in (t_i,t_{i+1}) ;\\
X_{t_i^+} = X_{t_i^-} + n(X_{t_i^-}) \xi ;
\end{cases}
\]
where $n(x)$ is the smallest integer $p$ such that $x + p \xi \notin A$ and where $t_{i+1}$ is the stopping time defined by the reaching time of $A$ by the process $(X_s)_{s\geq t_i}$. We recall that this interpretation is given in the case in which a unique jump is possible, a similar interpretation also exists in the case of a finite number of possible jumps.\\
\\
Finding solutions of the penalized equation (\ref{onejump}) does not require new techniques and is not a difficult question in itself. The majority of this part is concerned with building, uniform in $\epsilon$, a priori estimates (lemma \ref{estimate1}) to pass to the limit $ \epsilon \to 0$. Even though we use these uniform a priori estimates to prove proposition \ref{maxonejump} and theorem \ref{penalizedthm}, this estimate is crucial only to prove the existence of a limit as $\epsilon$ goes to $0$ (theorem \ref{limitsinglejump}).\\
\\
As explained in the introduction, we shall characterize in this part the solution of a Fokker-Planck equation (the limit density) of jumping particles with dual properties and not with a PDE. The main duality idea of this part is that a Fokker-Planck equation of jumping particle is somehow the dual or adjoint equation of a QVI, which by the way describes how it is optimal to jump and thus how dynamics of jumping particles evolve. Thus QVI are crucial for the study of the density of jumping particles. For the sake of clarity, the results needed on QVI are given in the appendix. We define here the notion of a smooth cost of jumps $k$. A function $k$ is said to be a smooth cost of jumps if it satisfies :
\begin{equation}\label{hypk}
\begin{cases}
\forall \xi \in K, k( \cdot, \xi) \in H^2(\mathbb{T}^d); \\
k^* :x \to \inf_{\xi \in K} k(x, \xi) \in W^{2, \infty}(\mathbb{T}^d);\\
\exists k_0 > 0 \text{ such that } k \geq k_0;
\end{cases}
\end{equation}
where $K$ is a finite subset of $\mathbb{T}^d$. The interpretation of $k(x,\xi)$ is that it is the cost paid by the player (or the energy used by a particle) to instantaneously go from $x$ to $x + \xi$. We also define the operator $M$ which plays an important role in the study of $QVI$ by :
\begin{equation}\label{defM}
M(k,u)(t,x) = \inf_{\xi \in K} \{k(x, \xi) + u(t, x + \xi)\}.
\end{equation}
When there is no ambiguity on $k$, we shall write only $Mu = M(k,u)$.
\\
\\
We begin the study of Fokker-Planck equations of jumping particles with the simpler case of a unique possible jump before addressing the case of a finite number of possible jumps. In each of these situations, we begin by constructing a penalized version of the problem and we then pass to the limit in the resulting penalized equation.

\section{The case of a unique possible jump}
We work here in the case in which a single jump $\xi$ is possible. We also assume that there is a measurable set $A$ on which the particles jump. The first part of this section is devoted to the study of a penalized equation. We then prove the existence and uniqueness of the limit density under a certain assumption on the set $A$. The study of the penalized equation is quite simple however we warn the reader that the estimate of lemma \ref{estimate1} that we use at a penalized level is crucial to study the limit case. An important feature of our model is that some assumption has to be made on the set on which the particles jump. This assumption shows in some sense the limit of this model. It can be formulated in the following way.

\begin{hyp}\label{hyp1}
The set $A$ is such that there exists $k \in L^{\infty}$ satisfying $(\ref{hypk})$, and $u \in L^2((0,T), H^2(\mathbb{T}^d))\cap H^1((0,T), L^2(\mathbb{T}^d)) \cap L^{\infty}((0,T) \times \mathbb{T}^d)$ such that 
\[
\begin{cases}
u(t,x) = k(x, \xi) + u(t, x + \xi) \text{ on } A;\\
u(T) = 0.
\end{cases}
\]
\end{hyp}
Formally, this assumption restricts the situations which we are able to model to a case in which the particles do not jump an infinite number of time in a finite time interval. Let us also note that if we interpret $u$ as the value function of some impulse control problem, then we are assuming that it is optimal to use its control on the set $A$. We advice the interested reader to look at the case $A = \mathbb{T}^d$ everywhere to convince himself/herself that the model we present is indeed not applicable to all measurable sets.

\subsection{ A penalized version of the problem}
In order to understand how the density of jumping particles behave we first introduce a smoother version of the problem. We here assume that the particles do not simply jump when they are in $A$ but that they have a given uniform probability of jumping in this set. This method allows us to work with a PDE. We naturally work with the equation :
\begin{equation*}{(2)}
\begin{cases}
\partial_t m_{\epsilon}(t,x) -  \nu\Delta m_{\epsilon}(t,x) + \frac{1}{\epsilon} \mathbb{1}_{A}m_{\epsilon}(t,x) - \frac{1}{\epsilon} (\mathbb{1}_{A}m_{\epsilon})(t,x- \xi) = 0 \text{ in } (0,T)\times \mathbb{T}^d; \\
m(0,x) = m_0(x) \text{ in } \mathbb{T}^d;
\end{cases}
\end{equation*}
where $\epsilon > 0$ is a real number which describes the probability of jumping in $A$. The term $ \frac{1}{\epsilon} \mathbb{1}_{A} m(t,x)$ stands for the leaving rate of particles which jump in $A$. The term $- \frac{1}{\epsilon}( \mathbb{1}_{A} m)(t,x- \xi)$ stands for the arriving rate of particles which jump at $(t,x - \xi)$ and thus which arrive at $(t,x)$. As $\epsilon$ goes to $0$, the probability of jumping becomes more and more important. Thus finding the limit as $\epsilon$ goes to $0$ gives the desired density of particles.\\
\\
For the rest of this paper we may not always write the "$(t,x)$" in order to lighten the notations. So by default, if the variable considered is not written, it is $(t,x)$. 
\\
\\
We begin by showing how we can find estimates on such a penalized equation. Let us introduce the set $H$ defined by
\begin{equation}\label{defH}
H := \{ v \in L^2((0,T), H^1(\mathbb{T}^d)), \partial_t v \in L^2((0,T), H^{-1}(\mathbb{T}^d))\}.
\end{equation}
The following lemma will be useful to establish a priori estimates on $m$.

\begin{Lemma}\label{comp1}
Let $m \in L^2((0,T), H^2(\mathbb{T}^d))$ be a solution of $(\ref{onejump})$ and $k$ satisfying $(\ref{hypk})$. If $m\geq 0$, $u \in H$ and $u\leq Mu$ almost everywhere on $A$, then 
\[
\int_0^T \int_{\mathbb{T}^d} (\partial_t m -  \nu\Delta m)u + \frac{1}{\epsilon} \int_{A}km \geq 0.
\]
\end{Lemma}

\begin{proof}
We multiply $(\ref{onejump})$ by $u$ and we integrate, we then obtain after a change of variable :
\[
\int_0^T \int_{\mathbb{T}^d} (\partial_t m -  \nu\Delta m)u + \frac{1}{\epsilon} \int_{A} m u = \frac{1}{\epsilon} \int_{A}m(t,x) u(t,x + \xi)dt dx.
\]
Using the fact that $u \leq Mu$ on $A$ we deduce the desired result.
\end{proof}

The previous result suggests to work with the set $\mathcal{H}(k)$ for some $k$ where $\mathcal{H}(k)$ is defined by :
\[
\mathcal{H}(k) := \{ m \in L^2((0,T), H^1(\mathbb{T}^d)), D(k,m) > - \infty\};
\]
where $D(k,m)$ is defined by
\[
D(k,m) := \inf \{ \int_0^T (-\partial_t u -  \nu\Delta u, m)_{H^{-1} \times H^1}  - \int_{\mathbb{T}^d}u(0) m_0 | u \in H, u \leq M(k,u) \text{ on } A, u(T) = 0 \}.
\]
We recall that $M$ is defined in (\ref{defM}). When no confusion is possible for $k$, as we do for $M$, we shall only write $D(m)$ and $\mathcal{H}$.\\
\\
The set $\mathcal{H}$ has to be interpreted as the set of admissible solutions of the Fokker-Planck equation. Indeed if $m$ is a density of jumping particles, then when particles leave (or jump) we should have exactly the same arriving particles $\xi$ further. We recall that $\partial_t m - \nu\Delta m \leq 0$ is interpreted as particles leaving and $\partial_t m - \nu\Delta m \geq 0$ as particles arriving. Thus it is natural to measure the variation of $\partial_t m -\nu \Delta m$ with functions $u$ such that $u(x) - u(x + \xi) \leq k$. This "test" quantifies the fact that some negativity for $\partial_t m - \nu\Delta m$ has to be compensate by some positivity of this quantity $\xi$ further. Moreover, as jumps can only occur on $A$ we restrict ourselves to the case in which those conditions are satisfied only on $A$. We now prove the following lemma, which states that in some sense, the quantity $D(m)$ is of interest to bound $m$ in some functional space (this lemma is crucial to study the limit $\epsilon \to 0$):

\begin{Lemma}\label{estimate1}
Let $k$ be such that it satisfies $(\ref{hypk})$ and hypothesis $1$ with a given $u \in H$. For any $ m \in H \cap \mathcal{H}(k)$, $m \geq 0$, there exists $C(k) >0$ depending only on $k$ and $||u||_{L^{\infty}}$ such that
\[
||m||^2_{L^2((0,T), H^1)} \leq -D(k,m) + C(k)(1 + ||m||_{L^2((0,T), H^1)} )||m_0||_{L^2}.
\]
\end{Lemma}
\begin{Rem}
The assumption $m \in H$ is crucial as it allows us to deal with the problem of the time regularity. Moreover, the constant $C(k)$ does not depend on $A$.
\end{Rem}
\begin{proof}
Because hypothesis $1$ is satisfied, we are able to apply proposition \ref{weakqvi} (appendix) and we deduce that there exists $\tilde{u} \in L^2((0,T), H^1(\mathbb{T}^d))$ such that :
\[
\begin{cases}
\tilde{u} \leq M(k,\tilde{u}) \text{ on } A;\\
\forall v \in H, v \leq M(k,\tilde{u}) \text{ on A};\\
-\int_0^T \int_{\mathbb{T}^d} \partial_t v (v -\tilde{u}) + \nu \int_0^T \int_{\mathbb{T}^d} \nabla \tilde{u} \cdot \nabla(v-\tilde{u}) + \frac{1}{2}\int_{\mathbb{T}^d}|v(T)|^2 \geq \int_0^T \int_{\mathbb{T}^d} (\Delta m) (v-\tilde{u}).
\end{cases}
\]
Because we made the assumption $m \in H$, we can remark that 
\[
D(k,m) =  \inf \{ \int_0^T (\partial_t m -  \nu\Delta m, v)_{H^{-1} \times H^1}  | v \in L^2((0,T), H^1(\mathbb{T}^d)), v \leq M(k,v) \text{ on } A\}.
\]
Thus we deduce
\begin{equation}\label{astep}
D(k,m) \leq \int_0^T (\partial_t m -  \nu\Delta m, \tilde{u})_{H^{-1} \times H^1}.
\end{equation}
One would like to write :
\[
\begin{aligned}
\int_0^T (\partial_t m -  \nu\Delta m, \tilde{u})_{H^{-1} \times H^1} &= \int_0^T (-\partial_t \tilde{u} - \nu \Delta \tilde{u}, m)_{H^{-1} \times H^1} - \int_{\mathbb{T}^d} \tilde{u}(0) m_0;\\
& \leq \int_0^T ( \Delta m, m)_{H^{-1} \times H^1} - \int_{\mathbb{T}^d}\tilde{u}(0) m_0;
\end{aligned}
\]
but since $\tilde{u} \notin H$, this does not make sense. However, because of the weak variational inequality satisfied by $\tilde{u}$, we can deduce that 
\[
\int_0^T (\partial_t m -  \nu\Delta m, \tilde{u})_{H^{-1} \times H^1} \leq \int_0^T (\Delta m, m)_{H^{-1} \times H^1} + ||m_0||_{L^2}||\tilde{u}||_{L^{\infty}(L^2)}.
\]
Hence we obtain that
\[
D(k,m) \leq - \int_0^T \int_{\mathbb{T}^d}|\nabla m|^2 + ||m_0||_{L^2}||\tilde{u}||_{L^{\infty}(L^2)}.
\]
Recalling the estimate of proposition \ref{weakqvi} (appendix), the result is proved.
\end{proof}

This lemma suggests to find a priori estimates for solutions of $(\ref{onejump})$ by looking at the quantity $D(m)$. However, this estimate requires the positivity of $m$. In order to use this estimate to exhibit solutions of $(\ref{onejump})$, we prove a maximum principle for this equation. This proof is very general and can be applied to more general equations. See \citep{lions2007cours} for an example of the use of this proof for systems of conservation laws for instance.

\begin{Prop}\label{maxonejump}
Let $\lambda \in L^{\infty}((0,T) \times \mathbb{T}^d)$, $\lambda \geq 0$, $m_0 \in L^2(\mathbb{T}^d)$, $m_0 \geq 0$, and $m \in L^2((0,T)\times \mathbb{T}^d)$ be a solution of
\begin{equation}\label{eqlambda}
\begin{cases}
\partial_t m(t,x) -  \nu\Delta m(t,x) + \lambda(t,x) m(t,x) - \lambda(t,x - \xi) m(t,x- \xi) = 0 \text{ in } (0,T) \times \mathbb{T}^d ;\\
m(0) = m_0 \text{ in } \mathbb{T}^d.
\end{cases}
\end{equation}
Then, $m \geq 0$.
\end{Prop}

\begin{proof}
We assume in a first time that $\lambda$ and $m_0$ are smooth functions and that $m_0 > 0$. Then by classical parabolic estimates, $m$ is also smooth ($C^1$ in time and $C^2$ in space). If there exists $(t_0, x_0) \in (0,T) \times \mathbb{T}^d$ such that $m(t_0, x_0) < 0$, then there exists $\delta > 0$ such that $m(t_0, x_0) + \delta t_0 < 0$. We define $µ$ by :
\[
µ(t,x) = m(t,x) + \delta t, \forall (t,x) \in (0,T) \times \mathbb{T}^d.
\]
For any $x\in \mathbb{T}^d, µ(0,x) > 0$ and $µ(t_0, x_0) < 0$, thus there exists $(t_1,x_1) \in (0,T) \times \mathbb{T}^d$ such that
\[
\begin{cases}
µ(t) \geq 0, \forall t \leq t_1; \\
µ(t_1,x_1) = 0 ;\\
\partial_t µ(t_1,x_1) \leq 0 ;\\
\Delta µ (t_1,x_1) \geq 0.
\end{cases}
\]
Let us remark that 
\[
\begin{aligned}
\partial_t µ &= \partial_t m + \delta ;\\
& = \delta +  \nu\Delta m - \lambda m + (\lambda m)(\cdot - \xi).
\end{aligned}
\]
Evaluating this last expression at $(t_1, x_1)$ we obtain that $\partial_t µ(t_1,x_1) > 0$ which is impossible. So we have proven that if $\lambda$ and $m_0$ are smooth, then $m \geq 0$. Because of the uniqueness of solutions of (\ref{eqlambda}) (which will be independently proved in the theorem $1.1$) this result extends to non smooth $\lambda$ and $m_0$ with only $m_0 \geq 0$.
\end{proof}

Next, we establish the main result of this section : the existence and uniqueness of a solution of $(\ref{onejump})$.

\begin{Theorem}\label{penalizedthm}
For any $m_0 \in L^2(\mathbb{T}^d)$, $m_0 \geq 0$, there exists a unique $m \in H$ such that 
\begin{equation*}{(2)}
\begin{cases}
\partial_t m -  \nu \Delta m + \frac{1}{\epsilon}\mathbb{1}_A(t,x) m(t,x) - \frac{1}{\epsilon}\mathbb{1}_A(t,x - \xi) m(t,x- \xi) = 0 \text{ in } (0,T) \times \mathbb{T}^d ;\\
m(0) = m_0 \text{ in } \mathbb{T}^d.
\end{cases}
\end{equation*}
where the first line has to be taken in the sense of distributions. Moreover, $m\geq 0$.
\end{Theorem}

\begin{proof}
We define $\lambda \in L^{\infty}$ by
\[
\lambda = \frac{1}{\epsilon} \mathbb{1}_A
\]
We then define the application $\mathcal{F}$ from $L^2((0,T)\times \mathbb{T}^d)$ to itself by : $\mathcal{F}(m)$ is the only solution in $H$ of
\[
\begin{cases}
\partial_t \mathcal{F}(m) -   \nu \Delta \mathcal{F}(m) + \lambda(t,x) m(t,x) - \lambda(t,x - \xi) m(t,x- \xi) = 0 \text{ in } (0,T) \times \mathbb{T}^d ;\\
m(0) = m_0 \text{ in } \mathbb{T}^d.
\end{cases}
\]
By standard parabolic estimates, $\mathcal{F}$ is continuous and compact. Let us take $µ \in [0,1]$ and $m \in L^2((0,T)\times \mathbb{T}^d)$ such that $m = µ \mathcal{F}(m)$, $m$ satisfies
\[
\begin{cases}
\partial_t m -  \nu \Delta m + µ\lambda(t,x) m(t,x) - µ\lambda(t,x - \xi) m(t,x- \xi) = 0 \text{ in } (0,T) \times \mathbb{T}^d ; \\
m(0) = m_0 \text{ in } \mathbb{T}^d.
\end{cases}
\]
By proposition $\ref{maxonejump}$, $m \geq 0$.  For any $v \in H$ such that $v \leq M v$ on $A$, we obtain using lemma $\ref{comp1}$:
\[
\int_0^T \int_{\mathbb{T}^d} (\partial_t m - \nu \Delta m)v \geq - µ \int_0^T \int_{\mathbb{T}^d} \lambda m.
\]
Hence by the lemma $\ref{estimate1}$ we deduce that there exists $C > 0$ independent of $µ$ such that :
\[
\begin{aligned}
||m||_{L^2((0,T), H^1(\mathbb{T}^d))}^2 & \leq C ||m||_{L^2((0,T),H^1(\mathbb{T}^d))} ||m_0||_{L^2} + µ \int_0^T \int_{\mathbb{T}^d} \lambda m; \\
& \leq C ||m||_{L^2((0,T),H^1(\mathbb{T}^d))} ||m_0||_{L^2} + µ ||\lambda||_{L^2} ||m||_{L^2}.
\end{aligned}
\]
From which we deduce that the set 
\[
\{ m \in L^2, \exists µ \in [0,1], m = µ \mathcal{F}(m)\}
\]
is bounded in $L^2$. Applying Schaefer's fixed point theorem, we obtain that there exists a solution of $(\ref{onejump})$.\\
\\
For any $m_1$ and $m_2$ solutions of $(\ref{onejump})$, we denote by $\delta m = m_1 - m_2$ the difference of these solutions. The function $\delta m$ satisfies
\[
\begin{cases}
\partial_t \delta m(t,x) -  \nu\Delta \delta m(t,x) + \lambda(t,x) \delta m(t,x) - \lambda(t,x - \xi) \delta m(t,x- \xi) = 0 \text{ in } (0,T) \times \mathbb{T}^d ; \\
\delta m(0) = 0 \text{ in } \mathbb{T}^d.
\end{cases}
\]
Multiplying this equation by $\delta m$ and integrating in space, we obtain
\[
\frac{1}{2} \frac{d}{dt} \int_{\mathbb{T}^d}( \delta m)^2 + \nu \int_{\mathbb{T}^d} |\nabla \delta m|^2 + \int_{\mathbb{T}^d} \lambda (\delta m)^2 = \int_{\mathbb{T}^d}(\lambda \delta m)(t,x) (\delta m)(t, x+ \xi) dx;
\]
\[
\frac{1}{2} \frac{d}{dt} \int_{\mathbb{T}^d} (\delta m)^2 \leq ||\lambda||_{L^{\infty}} || \delta m(t)||_{L^2(\mathbb{T}^d)}^2.
\]
Finally, we deduce, using Gronwall's lemma, that $\delta m = 0$ and thus that there exists a unique solution of $(\ref{onejump})$.

\end{proof}

\subsection{Existence of a limit density}
We show here how we can pass to the limit in the equation $(\ref{onejump})$ and hence obtain a characterization of the density of jumping particles.
For the rest of this section, we fix $k$ given in hypothesis $1$. We now describe the behavior of the solutions of $(\ref{onejump})$ as $\epsilon$ goes to $0$.

\begin{Theorem}\label{limitsinglejump}
Assume hypothesis $1$ holds, then there exists $m \in L^2((0,T), H^1(\mathbb{T}^d))$ such that, extracting a subsequence if necessary, $(m_{\epsilon})_{\epsilon}$ converges weakly in $L^2((0,T), H^1(\mathbb{T}^d))$ toward $m$ which satisfies
\[
\begin{cases}
D(m) > - \infty ;\\
m = 0 \text{ a.e. in } A.
\end{cases}
\]
Moreover, for all $v \in H$ such that $v \leq Mv$ on $A$ and $v(T) = 0$, 
\[
\begin{aligned}
&\int_0^T (-\partial_t v - \nu \Delta v, m)_{H^{-1} \times H^1} - \int_{\mathbb{T}^d}v(0)m_0 \\
&\geq \int_0^T (-\partial_t u -  \nu\Delta u, m)_{H^{-1} \times H^1} - \int_{\mathbb{T}^d}u(0)m_0= D(m);
\end{aligned}
\]
for any $u \in H$ which satisfies $u = Mu$ on $A$, $u(T) = 0$.
\end{Theorem}
\begin{Rem}
Let us note that the trace condition $m(0) = m_0$ is not satisfied, we send the reader to the paragraph following the proof for the interpretation of this fact.
\end{Rem}

\begin{proof}
We define $m_{\epsilon}$ for all $\epsilon > 0$ as the unique solution of
\[
\begin{cases}
\partial_t m_{\epsilon} -  \nu \Delta m_{\epsilon} + \frac{1}{\epsilon}\mathbb{1}_A(t,x) m_{\epsilon}(t,x) - \frac{1}{\epsilon}\mathbb{1}_A(t,x - \xi) m_{\epsilon}(t,x- \xi) = 0 \text{ in } (0,T) \times \mathbb{T}^d ;\\
m_{\epsilon}(0) = m_0 \text{ in } \mathbb{T}^d.
\end{cases}
\]
For all $\epsilon > 0$, in view of lemma $1.1$, we can observe that 
\begin{equation*}
-\infty < D(m_{\epsilon}) = \int_0^T (-\partial_tu -  \nu\Delta u, m_{\epsilon})_{H^{-1} \times H^1} - \int_{\mathbb{T}^d}u(0)m_0;
\end{equation*}
where $u \in L^{\infty} \cap L^2((0,T), H^1(\mathbb{T}^d))$ is given by hypothesis $1$. Hence we deduce from Lemma $1.2$ that $(m_{\epsilon})_{\epsilon > 0}$ is a bounded sequence of $L^2((0,T), H^1(\mathbb{T}^d))$ and as a consequence, that $(D(m_{\epsilon}))_{\epsilon > 0}$ is also a bounded sequence. Thus, extracting a subsequence if necessary, $(m_{\epsilon})_{\epsilon > 0}$ converges weakly toward a limit $m \in L^2((0,T), H^1(\mathbb{T}^d))$. Now take any $v \in H$ such that $v \leq Mv$ on $A$ and any $u\in H$ such that $u = Mu$ on $A$ and $u(T) = v(T) = 0$. For all $\epsilon > 0$ :
\[
\int_0^T (\partial_t m_{\epsilon},u)_{H^{-1} \times H^1} +  \nu \int_0^T \int_{\mathbb{T}^d} \nabla m_{\epsilon} \nabla u \leq \int_0^T (\partial_t m_{\epsilon},v)_{H^{-1} \times H^1} + \nu \int_0^T \int_{\mathbb{T}^d} \nabla m_{\epsilon} \nabla v.
\]
Thus we deduce that
\[
\begin{aligned}
\int_0^T (-\partial_t u - \nu \Delta u, m_{\epsilon})_{H^{-1} \times H^1}  - \int_{\mathbb{T}^d} u(0) m_0\\
\leq \int_0^T (-\partial_t v -  \nu\Delta v, m_{\epsilon})_{H^{-1} \times H^1} - \int_{\mathbb{T}^d}v(0) m_0.
\end{aligned}
\]
Passing to the limit $\epsilon$ goes to $0$ we deduce :
\[
\begin{aligned}
&\int_0^T (-\partial_t v -  \nu \Delta v, m)_{H^{-1} \times H^1} - \int_{\mathbb{T}^d}v(0)m_0 \\
&\geq \int_0^T (-\partial_t u -   \nu \Delta u, m)_{H^{-1} \times H^1} - \int_{\mathbb{T}^d}u(0)m_0.
\end{aligned}
\]
In particular,
\[
D(m) > -\infty.
\]
Let us note that for all $\epsilon > 0$
\[
D(m_{\epsilon}) = - \frac{1}{\epsilon} \int_A k m_{\epsilon} \leq - \frac{1}{\epsilon} \int_A k_0 m_{\epsilon}.
\]
Thus, $m = 0$ almost everywhere on $A$ because $(D(m_{\epsilon}))_{\epsilon}$ is bounded.
\end{proof}

\subsection{Interpretation of the limit density}
From a variational point of view,  the properties of the limit density $m$ given in this theorem are characterizing what we expect for such a density. Indeed, as we mentioned earlier, $D(m) > -\infty$ stands for the fact that $m$ is an admissible density for describing jumping particles. The condition $m = 0$ on $A$ stands for the fact that $m$ is a density of particles which are actually jumping on $A$ because otherwise there will be particles on $A$. The condition :
\begin{equation}\label{propD}
\begin{cases}
\forall u \in H, u = Mu \text{ on } A , u(T) = 0:\\
D(m) = \int_0^T (-\partial_t u -  \nu\Delta u,m)_{H^{-1} \times H^1} - \int_{\mathbb{T}^d} u(0) m_0;
\end{cases}
\end{equation}
stands fromally for the fact those particles are not jumping elsewhere than on $A$. Indeed at a penalized level we know that
\[
D(m_{\epsilon}) = -\frac{1}{\epsilon} \int_{A} km_{\epsilon}.
\]
Thus the quantity $D(m)$ is closely related to the set $A$ on which the particles actually jump. It appears that the quantity $D(m)$ measures the total aggregate costs "paid" by all the particles which jump. Because $m= 0$ on $A$, we know that the particles jump at least on $A$. Formally $(\ref{propD})$ states that the particles do not jump elsewhere than on $A$ ; because the total "cost" is minimum for particles which jump on at least $A$. To understand why we state that (\ref{propD}) stands for the fact that the particles do not jump elsewhere than on $A$, let us look at an example : we take $B$ a measurable subset of $\mathbb{T}^d$ which satisfies hypothesis $1$ and such that $A \subset B$. We define $µ_{\epsilon}$ by :
\[
\begin{cases}
\partial_t µ_{\epsilon} -  \nu\Delta µ_{\epsilon} + \frac{1}{\epsilon} \mathbb{1}_{B}µ_{\epsilon} - \frac{1}{\epsilon} (\mathbb{1}_{B}µ_{\epsilon})(t,x- \xi) = 0; \\
µ(0) = m_0.
\end{cases}
\]
Letting $\epsilon$ go to $0$, recalling the previous theorem, $µ$ satisfies $D(µ) > -\infty$ and $µ = 0$ on $A$ because $µ = 0$ on $B$. What is differentiating $µ$ from $m$ is that, for $v \in H, v \leq Mv, v = Mv$  on  $A$, we do not necessary have an equality in :
\[
D(µ) \leq \int_0^T (\partial_t µ,v)_{H^{-1} \times H^1} + \nu \int_0^T \int_{\mathbb{T}^d} \nabla µ \nabla v;
\]
if $v \ne Mv$ on $B$.\\
\\
Finally, let us note that the initial condition $m(0) = m_0$ may not be satisfied. This is a consequence of the fact that if $A$ is not negligible near $\{t = 0\}$, then particles are jumping instantaneously. We cannot expect in such a case for the initial condition to be satisfied. However the variational relation satisfied by $m$ is sufficient to "remember" that the density starts from $m_0$. In general, there is no time regularity for the solution $m$. This discussion leads to the following definition :

\begin{Def}\label{def1}
For any positive $m_0 \in L^2(\mathbb{T}^d)$, $A$ measurable subset of $\mathbb{T}^d$, $m \in L^2((0,T), H^1(\mathbb{T}^d))$ is called a solution of the Fokker-Planck equation of particles jumping on $A$ if 
\begin{itemize}
\item $m = 0$ on $A$;
\item \[
\begin{aligned}
\forall v, u\in H, v \leq Mv \text{ on } A, u = Mu \text{ on } A, v(T) = 0, u(T) = 0 :\\
 \int_0^T (-\partial_t (u-v) -  \nu\Delta (u-v),m)_{H^{-1} \times H^1} - \int_{\mathbb{T}^d} (u-v)(0) m_0 \leq 0.
\end{aligned}
\]
\end{itemize}
\end{Def}
Even though we just explain in which extent this definition is legitimate, the following section on the uniqueness, despite being a bit more restrictive on the set $A$, justifies this choice of definition.

\subsection{Uniqueness of the limit density}
We now discuss the uniqueness of such solutions. We state that the uniqueness holds under certain assumptions on the set $A$. Let us note that it is classical to have some assumptions on the domain in the study of parabolic PDE in time dependent domain, see \citep{gianazza1996abstract, calvo2017parabolic}. We make here the following assumption :
\begin{hyp}\label{hypunicite}
The set $A$ is either a closed set with Lipschitz boundary such that $\{T\} \times ( \mathbb{T}^d \cap A)$ is an open set with Lipschitz boundary, or it is non-decreasing in time (for the inclusion).
\end{hyp}

Our main argument is that the uniqueness of solutions of the Fokker-Planck equation can be deduced from an existence result for an "adjoint" equation. Under hypothesis \ref{hypunicite}, we have the following result :

\begin{Lemma}\label{lemmeunicite1}
Assume hypothesis \ref{hyp1} and \ref{hypunicite} hold and take $(k,u)$ given by hypothesis \ref{hyp1}. There exists $\epsilon > 0$ such that for any $f \in L^{\infty}((0,T) \times \mathbb{T}^d)$ $||f||_{L^{\infty}} \leq \epsilon$, there exists $v \in H$ such that, for all $µ \in L^2((0,T), H^1(\mathbb{T}^d))$, $µ = 0$ on $A$ : 
\[
\begin{cases}
\int_0^T(-\partial_t v -  \nu\Delta v,µ)_{H^{-1} \times H^1} = \int_0^T \int_{\mathbb{T}^d} (-\partial_t u -  \nu\Delta u + f,µ)_{H^{-1} \times H^1}  ;\\
v = M(k,v) \text{ on } A ;\\
v(T) = u(T).
\end{cases}
\]
\end{Lemma}

\begin{proof}
See appendix B, theorem \ref{existenceapp}.
\end{proof}
We are now able to prove the following result :
\begin{Theorem}\label{thmunicite1}
Assume hypothesis \ref{hyp1} and \ref{hypunicite} hold, for any non negative $m_0 \in L^2(\mathbb{T}^d)$ there exists at most one $m \in L^2((0,T), H^1(\mathbb{T}^d))$ solution of the problem in the sense of definition \ref{def1}.
\end{Theorem}

\begin{proof}
We denote by $m_1$ and $m_2$ two solutions. The idea of the proof consists in constructing for $i= 1,2$, $v_i \in H$ as in lemma \ref{lemmeunicite1} with respective second term $f = f_1$ and $f= f_2$, where $||f_1||_{L^{\infty}},||f_2||_{L^{\infty}} \leq \epsilon$. We then evaluate
\[
\int_0^T (-\partial_t (v_1 - v_2) - \nu \Delta (v_1 - v_2),m_1 - m_2)_{H^{-1}\times H^1}.
\]
This proof is the adaptation of the uniqueness proof we give for the MFG problem at the end of this paper, which is itself an adaptation of the proof of uniqueness of J.-M. Lasry and P.-L. Lions \citep{lasry2007mean}. Because $m_1$ and $m_2$ are solutions, the following holds:
\[
D(m_j) = \int_0^T (-\partial_t (v_i) - \nu \Delta (v_i),m_j)_{H^{-1}\times H^1} - \int_{\mathbb{T}^d}v_i(0)m_0.
\]
Hence,
\begin{equation}\label{intermediate}
\int_0^T (-\partial_t (v_1 - v_2) - \nu \Delta (v_1 - v_2),m_1 - m_2)_{H^{-1}\times H^1} = 0.
\end{equation}
On the other hand, using lemma \ref{lemmeunicite1}, we derive that
\[
\int_0^T (-\partial_t (v_1 - v_2) -  \nu\Delta (v_1 - v_2),m_1 - m_2)_{H^{-1}\times H^1} = \int_0^T\int_{\mathbb{T}^d}(m_1 - m_2) (f_1 - f_2).
\]
Recalling (\ref{intermediate}), we obtain that 
$$\int_0^T\int_{\mathbb{T}^d}(m_1 - m_2) (f_1 - f_2) = 0.$$
Thus we deduce that $m_1 = m_2$ because $m_1 - m_2$ is orthogonal to the ball of center $0$ and radius $2\epsilon$ of $L^{\infty}((0,T) \times \mathbb{T}^d)$.
\end{proof}

\subsection{A remark on the hypothesis for uniqueness}
Even though we are only able to prove the uniqueness of solutions of the Fokker-Planck equation under hypothesis \ref{hypunicite}, we conjecture that uniqueness is a more general property for this equation. Indeed we believe that an approach similar to the one we present in the second part of Appendix B (theorem \ref{uniqueapp}) can be adapted to this situation. The technical difficulty which we have not been able to overcome for the moment is to prove some time regularity for $m$ when the set $A$ is decreasing.\\
\\
However, we hope that the range of applications of hypothesis \ref{hypunicite} is large enough to convince the reader that the notion of solution of the Fokker-Planck equation we present is the correct one.

\section{ The case of a finite number of possible jumps}
We now address a more general model as we look at situations in which different jumps can occur. As we shall see, all the results of the case of a single jump are adaptable to the case of a finite number of jumps. However there are in this section more notations and we advise not to read this section before the previous one. We denote by $K \subset \mathbb{T}^d$ the finite set of possible jumps. In this setting a single set $A$ is no longer sufficient to describe all the jumps. We introduce $V$ which describes the jumps by : $ V(\xi, t,x)$ is the proportion of particles which use the jump $\xi$ at $(t,x)$. We assume the following :
\begin{equation}\label{hypV}
\begin{cases}
V \in L^{\infty}(K,(0,T),\mathbb{T}^d) ;\\
V \geq 0 ;\\
\sum_{\xi \in K} V(\xi,t,x) \leq 1.
\end{cases}
\end{equation}
We also define the following sets :
\[
\forall \xi \in K, A_{\xi} := \{ V(\xi, \cdot , \cdot ) > 0 \};
\]
\[
A = \cup_{\xi \in K} A_{\xi}.
\]
As in the case of a single jump, an assumption on the sets on which the particles jump is still needed. We make here the following assumption :
\begin{hyp}\label{hypmultiple}
There exists $k$ satisfying $(\ref{hypk})$ and $u \in H$ such that
\begin{equation}\label{hypu}
\begin{cases}
u \leq M(k,u) \text{ in } (0,T) \times \mathbb{T}^d;\\
u(T) = 0; \\
\forall \xi \in K, V(\xi,t,x) (k(x,\xi) + u(t, x+ \xi) - u(t,x)) = 0 \text{ a.e. in } (0,T) \times \mathbb{T}^d .
\end{cases}
\end{equation}
\end{hyp}
We recall that the set $H$ is defined by $$H := \{ v \in L^2((0,T), H^1(\mathbb{T}^d)), \partial_t v \in L^2((0,T), H^{-1}(\mathbb{T}^d))\}.$$
This hypothesis is slightly more sophisticated than hypothesis $1$. This is due to the fact that multiple jumps being possible, we have to be more precise. We still assume that the sets on which the particles are jumping are given as a result of an impulse control problem but we precise for which $\xi$ the minimum is reached.

\subsection{The penalized equation}
We introduce first a penalized version of the problem. We recall that this penalization models situations in which the particles have a certain probability to jump in the prescribed sets, and that the limit as $\epsilon$ goes to $0$ corresponds to the probability of jumping going to $1$. At this penalized level, we expect the density of particles $m$ to satisfy :

\begin{equation}\label{penalmultiple}
\begin{cases}
\partial_t m_{\epsilon} -  \nu\Delta m_{\epsilon} + \frac{1}{\epsilon} m_{\epsilon}(t,x)\big{(}\sum_{\xi \in K}V(\xi,t,x) \big{)}- \frac{1}{\epsilon} \sum_{\xi \in K} m_{\epsilon}(t,x- \xi) V(\xi,t, x - \xi) = 0\\
\text{ in } (0,T) \times \mathbb{T}^d ;\\
m(0) = m_0 \text{ in } \mathbb{T}^d.
\end{cases}
\end{equation}
We define the set $\mathcal{AD}(k)$ by :
\[
\mathcal{AD}(k) := \{ u \in H, u(T) = 0, \forall \xi \in K, V(t,x,\xi)(u(t,x) - k(x,\xi) - u(t,x + \xi)) \leq 0 \text{ on } A\}.
\]
This set represents the set of admissible solutions of an impulse control problem in which one can only use the jump $\xi \in K$ at $(t,x)$ if $V(t,x,\xi) > 0$. Let us assume that $m$ is a smooth solution of the previous PDE. Then, as in the case of a single jump, for any $u \in \mathcal{AD}(k)$, after a simple change of variable we obtain that :
\[
\begin{aligned}
\int_0^T (\partial_t m -  \nu\Delta m, u)_{H^{-1} \times H^1} &= - \frac{1}{\epsilon} \int_0^T \int_{\mathbb{T}^d}m(t,x) (\sum_{\xi \in K} V(\xi, t,x) (u(t,x) - u(t,x + \xi))) dtdx; \\
& \geq - \frac{1}{\epsilon} \int_0^T \int_{\mathbb{T}^d}m(t,x) (\sum_{\xi \in K} V(\xi, t,x)k(t,x,\xi)) dt dx;
\end{aligned}
\]
This leads us to define the meaningful quantity : 
\begin{equation}\label{D(m)}
D(k,m) := \inf \{ \int_0^T (-\partial_t u - \nu \Delta u, m)_{H^{-1} \times H^1} - \int_{\mathbb{T}^d} m_0 u(0) | u \in \mathcal{AD}(k)  \};
\end{equation}
We introduce the set $\mathcal{H}(k) := \{ m \in L^2((0,T), H^1(\mathbb{T}^d)) , D(k,m) > - \infty\}$. Let us remark that for any $u$ satisfying (\ref{hypu}) :
\[
D(k,m_{\epsilon}) = \int_0^T (\partial_t m_{\epsilon} - \nu \Delta m_{\epsilon}, u)_{H^{-1} \times H^1}.
\]
The proofs of the following lemmata are the exact analogous of the proof we did earlier in the case of a single jump so we do not present them here.

\begin{Lemma}\label{estimate2}
Assume hypothesis \ref{hypmultiple} holds for $(k,u)$. Then there exists $C > 0$ (independent on $m$ and depending on $k$ and $||u||_{L^{\infty}}$) such that for any $ m \in \mathcal{H}(k) \cap H$, with $m \geq 0$ :
\[
||m||^2_{L^2((0,T), H^1)} \leq - D(k,m) + C ||m||_{L^2((0,T), H^1)} ||m_0||_{L^2}.
\]
\end{Lemma}

\begin{Lemma}\label{max2}
Let $m_0 \in L^2(\mathbb{T}^d)$, $m_0 \geq 0$, and $m \in L^2((0,T), H^1(\mathbb{T}^d))$ be a solution of (\ref{penalmultiple}).
Then, $m \geq 0$.
\end{Lemma}

Furthermore, as in the case of a single jump, we can prove the following result.
\begin{Theorem}
Assume hypothesis \ref{hypmultiple} holds, then for all $m_0 \in L^2(\mathbb{T}^d)$ there exists a unique solution $m \in H$ of the penalized equation (\ref{penalmultiple}).
\end{Theorem}

\subsection{The limit density}
We now present how we can pass to the limit in (\ref{penalmultiple}) using the previous result. As in the case of a single jump, the existence follows from lemma \ref{estimate2} and we prove uniqueness under a more restrictive assumption to avoid technical difficulties.

\begin{Theorem}
If there exists $k$ satisfying (\ref{hypk}) such that hypothesis \ref{hypmultiple} holds, then there exists $m \in L^2((0,T), H^1(\mathbb{T}^d))$ such that 

\[
\begin{cases}
D(k,m) > -\infty ;\\
m = 0 \text{ in } A;
\end{cases}
\]
and, for all $v \in \mathcal{AD}(k) $ :
\[
\int_0^T (-\partial_t v -  \nu\Delta v, m )_{H^{-1} \times H^1} - \int_{\mathbb{T}^d} v(0)m_0 \geq \int_0^T (-\partial_t u - \nu \Delta u, m)_{H^{-1} \times H^1} - \int_{\mathbb{T}^d}u(0)m_0 = D(k,m);
\]
where $u\in H$ satisfies $(\ref{hypu})$.
\end{Theorem}
We recall that $A$ is defined by 
\[
A = \cup_{\xi \in K} \{V(\xi, \cdot, \cdot) > 0\}.
\]
We do not present the proof of this result as it is the same as the one we presented in the case of a unique jump. Like we did in the case of a single jump, we give the following definition :

\begin{Def}
For any $m_0 \in L^2(\mathbb{T}^d)$, $V$ measurable function satisfying $(\ref{hypV})$, $m \in L^2((0,T), H^1(\mathbb{T}^d))$ is called a solution of the Fokker-Planck equation of particles jumping with jumps described by $V$ if 
\begin{itemize}
\item $m = 0$ on $A$;
\item \[
\begin{cases}
\forall v \in H, v \leq Mv, \forall u \in H \text{ satisfying (\ref{hypu})} :\\
\int_0^T (-\partial_t v - \nu \Delta v, m )_{H^{-1} \times H^1} - \int_{\mathbb{T}^d} v(0)m_0 \geq \int_0^T (-\partial_t u - \nu \Delta u, m)_{H^{-1} \times H^1} - \int_{\mathbb{T}^d}v(0)m_0 = D(m).
\end{cases}
\]
\end{itemize}
\end{Def}

We now turn to the question of the uniqueness of such solutions. As in the case of a single jump, uniqueness is a consequence of an existence result for an adjoint equation. As the manner in which we proceed in the particular case of a single jump, we are going to make some assumption on $A$. Moreover here we add an hypothesis on the function $u$ given by hypothesis \ref{hypmultiple}. We state the following hypothesis :
\begin{hyp}\label{hypuniquemultiple}
The set $A$ is either a closed set with Lipschitz boundary such that $\{T\} \times (\mathbb{T}^d \cap A)$ is an open set with Lipschitz boundary, or it is non-decreasing in time (for the inclusion). Moreover, there exists $(k,u)$ satisfying hypothesis \ref{hypmultiple} such that :
\begin{equation}\label{separatexi}
\begin{cases}
\forall \xi \in K, x \in A_{\xi} \Rightarrow (k(x,\xi) + u(t,x+ \xi) - u(t,x)) = 0 ;\\
\forall \xi \in K, x \in A_{\xi} \Rightarrow \forall \xi' \ne \xi, u(t,x) < k(x,\xi') + u(t,x + \xi').
\end{cases}
\end{equation}
\end{hyp}
We are now able to state the uniqueness of solutions of Fokker-Planck equation of jumping particles in the case of a finite number of possible jumps for a set $A$ verifying hypothesis \ref{hypuniquemultiple}.

\begin{Theorem}
Under hypothesis \ref{hypuniquemultiple}, for any positive $m_0 \in L^2(\mathbb{T}^d)$, there exists at most one $m \in L^2((0,T), H^1(\mathbb{T}^d))$ solution of the Fokker-Planck equation of particles jumping with jumps described by $V$.
\end{Theorem}
This result is proved following the same argument as the one for a unique possible jump.

\section{The stationary case}
We now turn to the question of the stationary setting. Most of the arguments of the proofs of the results below follow the ones from the time dependent setting. Thus, we only details the arguments which differ form the one in the time dependent case.\\
\\
We assume that there exists $V$ such that :
\begin{equation*}
\begin{cases}
V \in L^{\infty}(K,\mathbb{T}^d) ;\\
V \geq 0 ;\\
\sum_{\xi \in K} V(\xi,x) \leq 1.
\end{cases}
\end{equation*} 
We denote by $K \subset \mathbb{T}^d$ the finite set of possible jumps. We assume that the following assumption is satisfied :
\begin{hyp}\label{hypmultiple2}
There exists $u \in H^2(\mathbb{T}^d) \cap L^{\infty}$ sand $k \in L^{\infty}(K \times \mathbb{T}^d)$ such that
\begin{equation*}
\begin{cases}
k \geq k_0 > 0;\\
\forall \xi \in K, k(\xi, \cdot) \in H^2(\mathbb{T}^d);\\
\forall \xi \in K, V(\xi,x) \left( k(x,\xi) + u( x+ \xi) - u(x) \right) = 0 \text{ a.e. in } \mathbb{T}^d ;
\end{cases}
\end{equation*}
where $M(k,u)$ is defined by :
\[
M(k,u)(x) = \inf_{\xi \in K} k(\xi, x) + u(x + \xi).
\]
\end{hyp}
We study here a stationary Fokker-Planck equation in which there is a fixed leaving rate of players $\delta > 0$ and a constant entry of players $\rho \in L^{\infty}(\mathbb{T}^d)$, $\rho \geq 0$. Namely, at a penalized level, we are interested in :
\begin{equation}\label{penalstat}
-\nu \Delta m_{\epsilon} + \delta m_{\epsilon} + \frac{1}{\epsilon} m_{\epsilon}(x)\big{(}\sum_{\xi \in K}V(\xi,x) \big{)}- \frac{1}{\epsilon} \sum_{\xi \in K} m_{\epsilon}(x- \xi) V(\xi, x - \xi) = \rho \text{ in } \mathbb{T}^d.
\end{equation}
This section is organized as follow : we first show the existence and uniqueness of solutions of the penalized equation (\ref{penalstat}). We then show the existence of a limit as $\epsilon$ goes to $0$ which satisfies the Fokker-Planck equation in a weak sense. We then prove the uniqueness of such limits.

\subsection{The penalized equation}
We begin this section by showing a general uniqueness result for equations of the type of (\ref{penalstat}).
\begin{Prop}
Let $(\lambda_1, .. , \lambda_n) \in L^{\infty}(\mathbb{T}^d)^n$. Then, for any $\rho \in L^2(\mathbb{T}^d)$, there exists at most one solution $m \in H^2(\mathbb{T}^d)$ of :
\begin{equation}
-\nu \Delta m + \delta m + m(x)\big{(}\sum_{i = 1}^n\lambda_i(x) \big{)}- \sum_{i =1}^n m(x- \xi) \lambda_i( x - \xi) = \rho \text{ in } \mathbb{T}^d.
\end{equation}
\end{Prop}
\begin{proof}
Let us assume that there exist two such solutions $m_1$ and $m_2$. Then if we note $µ= m_1 - m_2$, it satisfies :
\[
-\nu \Delta µ + \delta µ + µ(x)\big{(}\sum_{i = 1}^n\lambda_i(x) \big{)}- \sum_{i =1}^n µ(x- \xi) \lambda_i( x - \xi) = 0 \text{ in } \mathbb{T}^d.
\]
Let us assume $µ \ne 0$, then the operator $\mathcal{T}$ has a spectral radius $r \geq 1$, where $\mathcal{T}$ is defined from $L^1(\mathbb{T}^d)$ into itself as follows : for any $m \in L^1(\mathbb{T}^d)$, $\mathcal{T}(m)$ is the only solution of
\[
-\nu \Delta \mathcal{T}(m) + \delta \mathcal{T}(m) + \mathcal{T}(m)(x)\big{(}\sum_{i = 1}^n\lambda_i(x) \big{)}= \sum_{i =1}^n m(x- \xi) \lambda_i( x - \xi)  \text{ in } \mathbb{T}^d.
\]
The operator $\mathcal{T}$ is compact and $$\mathcal{T}(\{ m \in L^1(\mathbb{T}^d), m \geq 0\}) \subset \{ m \in L^1(\mathbb{T}^d), m \geq 0\}.$$ Thus by the Krein-Rutman theorem, there exists $r \geq 1$, $w \in \{ m \in L^1(\mathbb{T}^d), m \geq 0\}, w \ne 0$ such that :
\[
-\nu \Delta w + \delta w + w\big{(}\sum_{i = 1}^n\lambda_i(x) \big{)}= \frac{1}{r}\sum_{i =1}^n w(x- \xi) \lambda_i( x - \xi)  \text{ in } \mathbb{T}^d.
\]
Integrating in space, we obtain that
\[
\int_{\mathbb{T}^d}\delta w + \int_{\mathbb{T}^d}\left( 1 - \frac{1}{r}\right)(\sum_{i = 1}^n\lambda_i )w= 0.
\]
The two terms of the left hand side are positive so we deduce that $w = 0$, which is a contradiction. Thus $µ = 0$ and there exists a unique solution of this PDE.
\end{proof}
We now show that there exists a solution of (\ref{penalstat}).
\begin{Prop}
For any $\epsilon > 0$, there exists a unique solution $m \in L^1(\mathbb{T}^d)$ of (\ref{penalstat}). This solution $m$ is positive.
\end{Prop}
\begin{proof}
We define the application $\mathcal{T}$ from $L^1(\mathbb{T}^d)$ into itself by :  for any $m \in L^1(\mathbb{T}^d)$, $\mathcal{T}(m)$ is the only solution of
\[
-\nu \Delta \mathcal{T}(m) + \delta \mathcal{T}(m) + \mathcal{T}(m)(x)\big{(}\sum_{\xi \in K}V(\xi,x) \big{)}= \sum_{\xi \in K} m(x- \xi) V(\xi, x - \xi) + \rho  \text{ in } \mathbb{T}^d.
\]
Let us observe that if $m \geq 0$, then by the maximum principle $\mathcal{T}(m) \geq 0$. Moreover, if $m \in \Omega$ defined by :
\[
\Omega := \{ m \in L^1(\mathbb{T}^d), m \geq 0, \int_{\mathbb{T}^d} m(x)\left(\sum_{\xi \in K}V(\xi,x)\right)dx \leq \delta^{-1} \int_{\mathbb{T}^d} \rho\};
\]
then $\mathcal{T}(m) \in \Omega$. Indeed, integrating in space the equation which defines $\mathcal{T}(m)$, we obtain that
\[
\delta \int_{\mathbb{T}^d}\mathcal{T}(m) + \frac{1}{\epsilon} \int_{\mathbb{T}^d} \mathcal{T}(m) \left( \sum_{\xi \in K} V(\xi) \right) = \int_{\mathbb{T}^d}\rho + \frac{1}{\epsilon} \int_{\mathbb{T}^d} m\left(\sum_{\xi \in K}V(\xi)\right).
\]
Thus, $\Omega$ is stable by $\mathcal{T}$. Applying Schauder's fixed point theorem, we deduce that $\mathcal{T}$ has a fixed point. The uniqueness is given by the previous proposition.
\end{proof}
\begin{Rem}
This result does not depend on the sets on which the particles jump, i.e. it does not depend on the function $V$ except for the fact that $V \geq 0$ and $\sum_{\xi \in K} V(\xi) \leq 1$.
\end{Rem}

\subsection{Existence and uniqueness of the solution of the stationary Fokker-Planck equaiton}
We now turn to the existence of a limit density as $\epsilon$ goes to $0$. We define the following set:
\[
\mathcal{AD}(k) := \{ v \in H^1(\mathbb{T}^d), \forall \xi \in K, V(t,x,\xi)(v(x) - k(x,\xi) - v(x + \xi)) \leq 0 \text{ on } A\}.
\]
A crucial result to pass to the limit  is the following :
\begin{Lemma}\label{eststat}
Assume hypothesis \ref{hypmultiple2} holds true. Let $m\in H^1(\mathbb{T}^d)$ be such that :
\[
\begin{cases}
m \geq 0,\\
\delta \int_{\mathbb{T}^d} m = \int_{\mathbb{T}^d} \rho.
\end{cases}
\]
Then for any $k$ satisfying (\ref{hypk}) :
\[
||m||_{H^1(\mathbb{T}^d)}^2 \leq C\left(-D(k,m) + ||m||_{H^1}||\rho||_{L^2}\right);
\]
where $D(k,m)$ is defined by:
\[
D(k,m) = \{\nu \int_{\mathbb{T}^d} \nabla m \cdot \nabla v + \delta \int_{\mathbb{T}^d}vm - \int_{\mathbb{T}^d} \rho v | v \in \mathcal{AD}(k) \}.
\]
\end{Lemma}
The following result holds true :
\begin{Theorem}
Assume hypothesis \ref{hypmultiple2} holds true. Then, there exists a unique $m \in H^1(\mathbb{T}^d)$ such that :
\begin{itemize}
\item $\forall \xi \in K : V(\xi,x)m(x) = 0$ almost everywhere in $\mathbb{T}^d$.
\item For any $u \in H^1(\mathbb{T}^d)$ which satisfies
\[
\forall \xi \in K, V(\xi,x) \left( k(x,\xi) + u( x+ \xi) - u(x) \right) = 0 \text{ a.e. in } \mathbb{T}^d ;
\]
the following holds :
\[
\begin{aligned}
&\forall v \in \mathcal{AD}(k) :\\
&\nu \int_{\mathbb{T}^d}\nabla m \cdot \nabla (v-u) + \delta \int_{\mathbb{T}^d}m(v-u) \geq \int_{\mathbb{T}^d}\rho(v-u).
\end{aligned}
\]
\end{itemize}
\end{Theorem}
\begin{proof}
The proof of both existence and uniqueness are the exact analogous of the ones in the time dependent case, see proposition \ref{weakqvi2} in the appendix for the analogous in the stationary setting of proposition \ref{weakqvi}.
\end{proof}
\begin{Rem}
Let us note that the question of the uniqueness of solutions is a lot more simpler in the stationary case. Indeed the time regularity is no longer a problem and the existence of solutions for the "adjoint" problem is then easily proved following the argument of the time-dependent case.
\end{Rem}
\section{A remark on the generality of this method}
Before using this notion of solution of Fokker-Planck equation in a MFG of impulse control, we precise some straightforward generalizations of the results of the previous part. First working on the torus $\mathbb{T}^d$ does not play any role but simplifying the notations and fixing a framework. Thus those results generalize to more complex domain and boundary conditions. Secondly, the cost of jumps $k$ can be allowed to depend on the time variable. If this dependence is smooth, this does not change our results. Also more general jumps can be model with this kind of method. For example one can think of impulse control in which any jump is possible ; or a problem of optimal stopping time type, except that instead of leaving, "stopping" the trajectories restarts it at the origin (or at any given point). In this second problem, it is not the jumps but the destination which belongs to a finite set. Both of these optimization problems have value functions which solves QVI under some assumptions. In these two examples continuity of the value function is crucial, hence appropriate assumptions have to be made on the regularity of solutions and a solution of the Fokker-Planck equation may not be more regular than a measure. The important point is that with the QVI comes a notion of "admissible" solutions for the QVI (in this article being admissible is satisfying $u \leq Mu$). With this notion comes the notion of admissible density of particles which is, in this article, the fact that $D(m) > -\infty$. Then a priori estimates on the solutions of the Fokker-Planck equation are available and we can continue developing such solutions.\\
\\
Finally, let us note that optimal switching problems can be formulated in terms of QVI, see \citep{bensoussan1984impulse}. Thus we can model a density of particles, whose trajectories are given by optimal trajectories of an optimal switching problem, using the same technique as the one we have just presented in this section.

\part{ Mean field games of impulse control through quasi-variational inequalities}

We present in this part an application of the notion of solutions of a Fokker-Planck equation of jumping particles. We study a MFG of impulse control where the density of players is naturally a solution of this kind of equation. We work here in the case of a finite number of possible jumps. We denote by $K$ the set of jumps. We denote by $k$ satisfying $(\ref{hypk})$ the cost of the different jumps depending on the position. For any $v \in L^2((0,T), H^1(\mathbb{T}^d))$, we define $Mv$ by 
\[
Mv(t,x) = \inf_{\xi \in K} \{ k(x, \xi) + v(t,x+ \xi)\}.
\]
We denote by $f$ the running cost of the problem. The function $f$ depends on space, time and on the repartition of the player (i.e. $f = f(t,x,m)$). We make the following assumptions on $f$ :
\begin{itemize}
\item $f$ is continuous from $L^2((0,T), H^1(\mathbb{T}^d))$ endowed with its weak topology to $L^2((0,T), H^{-1}( \mathbb{T}^d))$.
\item  $f$ is uniformly bounded by below by a constant $-C$ (where $C>0$) on the positive elements of $L^2((0,T), H^1(\mathbb{T}^d))$
\item  $f$ maps $L^2((0,T), H^1(\mathbb{T}^d))$ to a bounded set of $L^p((0,T) \times \mathbb{T}^d)$ with $p > d$.
\end{itemize}
\begin{Rem}
The last assumption on $f$ can be replaced by any assumption which yields a uniform bound by above for the solution of the heat equation with source $f(m)$.
\end{Rem}
As in the previous part, we work on the $d$ dimensional torus to simplify the notations but all the following results are adaptable to more complex situations. We once again use the notation
\[
H = \{ v \in L^2((0,T), H^1(\mathbb{T}^d)) , \partial_t v \in L^2((0,T), H^{-1}(\mathbb{T}^d))\}.
\]
The problem we are interested in, is finding $(u,m)$ such that :

\begin{equation}\label{MFGsystem}
\begin{cases}
\max(-\partial_t u -  \nu \Delta u - f(m), u - Mu) = 0 \text{ in } (0,T) \times \mathbb{T}^d;\\
u(T) = 0 \text{ in } \mathbb{T}^d;\\
D(m) > - \infty; \\
\forall v \in H, v \leq Mv, v(T) = 0 :\\
\int_0^T(-\partial_t (v-u) -  \nu\Delta( v-u), m)_{H^{-1} \times H^1} - \int_{\mathbb{T}^d}(v-u)(0)m_0 \geq 0;\\
\int_0^T \int_{\mathbb{T}^d} (-\partial_t u -  \nu \Delta u -f(m))m = 0;
\end{cases}
\end{equation}
where $D(m)$ is defined by 
\begin{equation}\label{defDMFG}
D(m) := \inf_{v \in H_{ad}} \int_0^T (-\partial_t v -  \nu\Delta v, m)_{H^{-1} \times H^1} - \int_{\mathbb{T}^d}m_0 v(0);
\end{equation}
where $H_{ad} := \{ v \in H, v \leq Mv, v(T) = 0 \}$. The function $u$ denotes the value function of the impulse control problem for a generic player of the MFG and $m$ is the density of players.The first two lines have to be taken in the sense that $u$ is the solution of the associated QVI. Thus that it formally solves the impulse control problem for the generic player in which the running cost is $f(m)$ and $k$ is the cost for the jumps. We refer to the appendix for some details on $QVI$ and to \citep{bensoussan1984impulse} for a complete study of the problem. In view of the previous part, $m$ is a solution of a Fokker -Planck equation which models the density of (jumping) players of the game. Let us note that because there is no constraint such that $m = 0$ on $\{ u = Mu\}$, $m$ is not necessary the solution of a limit problem of the previous part with some $V$ (describing the jumps) well chosen depending on $u$. Indeed, in view of \citep{bertucci2017optimal}, we do not expect the existence of solutions if we impose such strong conditions which are assimilated with  Nash equilibria in pure strategies for the MFG. That is why we do not impose the condition $ m = 0 $ on $\{ u = Mu\}$ but the integral relation of the last line. It is the formulation of the fact that $(u,m)$ is a mixed solution of the MFG, i.e. that this system characterizes Nash equilibria in mixed strategies. We recall the interpretation of such a relation.\\
\\
Formally, a natural requirement for the solution of such a MFG shall be to impose that $m = 0 $ on $\{u = Mu\}$ which is the set where it is optimal to use an impulse control. The integral formulation in this system requires that $m =0$ on $\{ -\partial_t u - \Delta u < f(m)\}$ which is the set where it is strictly optimal to use an impulse control. The difference here is that on $\{ u = Mu\}$ one can still have $-\partial_t u - \Delta u = f(m)$ and thus that it is both optimal to stay and to use a control. We do not impose that $m$ vanishes in such a situation. Such a relaxation makes the problem more convex and allows us to prove an existence result while still conserving a uniqueness property.\\
\\
The methodology to work on $(\ref{MFGsystem})$ is the following : we first introduce a penalized version of this problem and then we show how we can pass to the limit to obtain the existence of solutions of $(\ref{MFGsystem})$. Later on we prove a result of uniqueness for such solutions.

\section{The penalized problem}
We introduce here the penalized problem :

\begin{equation}\label{penalizedMFG}
\begin{cases}
\max(-\partial_t u -  \nu\Delta u - f(m), u - Mu) = 0 \text{ in } (0,T) \times \mathbb{T}^d; \\
u(T) = 0 \text{ in } \mathbb{T}^d;\\
\partial_t m - \nu \Delta m + \frac{\alpha}{\epsilon} \mathbb{1}_{\{u = Mu\}} m - \sum_{\xi \in K} V(t,x- \xi,\xi) \frac{\alpha(t, x - \xi)}{\epsilon} \mathbb{1}_{\{u = Mu\}}(t, x - \xi) m(t, x- \xi)= 0 \\
\text{ in } (0,T)\times \mathbb{T}^d ;\\
m(0) = m_0 \text{ in } \mathbb{T}^d;\\
(u -Mu) + (-\partial_t u - \nu \Delta u - f(m)) \ne 0 \Rightarrow \alpha = 1 ;\\
\forall \xi \in K, u(t,x) \ne k(t,x,\xi) + u(t,x + \xi) \Rightarrow V(t,x,\xi) = 0 ; \\
\forall \xi \in K, \forall (t,x) \in (0,T) \times \mathbb{T}^d, 0 \leq V(t,x,\xi) \leq 1;\\
\forall (t,x) \in (0,T) \times \mathbb{T}^d, 0 \leq \alpha (t,x) \leq 1.
\end{cases}
\end{equation}
Recalling the previous part, it is natural to introduce first such a penalized system, and then pass to the limit $\epsilon$ goes to $0$. Indeed the equation satisfied by $m$ cannot be easily written in terms of a partial differential equation whereas it can at a penalized level. The potential $V$ gives at each point $(t,x)$ the jump used by the players at this point. If $V(t,x,\xi) \ne 0$ then some players use the jump $\xi$ at $(t,x)$. The term $\alpha$ adds convexity to the problem and make possible the existence of a solution ; it is the translation at a penalized level that we are looking for Nash equilibria of the MFG in mixed strategies. The technique here is inspired from \citep{bertucci2017optimal} where it is shown that such a system, for variational inequalities instead of quasi variational inequalities, leads to a solution of the MFG system. 

\begin{Theorem}
There exists a solution $(u,m, \alpha, V) \in (L^2((0,T), H^2(\mathbb{T}^d)) \cap H^1((0,T), L^2(\mathbb{T}^d))) \times L^2((0,T), H^1(\mathbb{T}^d)) \times L^{\infty}((0,T) \times \mathbb{T}^d) \times L^{\infty}(K \times (0,T) \times \mathbb{T}^d) $ of $(\ref{penalizedMFG})$ such that $u$ solves the associated QVI and the equation in $m$ is satisfied in a weak sense.
\end{Theorem}

\begin{proof}
We define an application $\mathcal{F}_1$ from $L^2((0,T),H^1(\mathbb{T}^d))$ (endowed with its weak topology) into itself (endowed with its strong topology) by : for any $m \in L^2((0,T), H^1(\mathbb{T}^d))$, $\mathcal{F}_1(m)$ is the only solution of the QVI with costs $f(m)$ and $k$. $\mathcal{F}_1$ is well defined and continuous recalling results on QVI (see appendix). Then, we define the set-valued function $\mathcal{F}_2$ from $L^2((0,T) \times \mathbb{T}^d)$ to $L^2((0,T), H^1(\mathbb{T}^d))$ (endowed with its weak topology) by : for any $u \in L^2$, 

\[
\mathcal{F}_2(u) := \{ m \in L^2((0,T),H^1( \mathbb{T}^d)), \exists \alpha, V \in L^{\infty}, (m,\alpha,V) \text{ solves } (\ref{penalizedMFG})(u)\};
\]
where $(\ref{penalizedMFG})(u)$ is
\begin{equation*}{(\ref{penalizedMFG})(u)}
\begin{cases}
\begin{aligned}
\partial_t m &-  \nu\Delta m  - \sum_{\xi \in K} V(t,x- \xi,\xi) \frac{\alpha(t, x - \xi)}{\epsilon} \mathbb{1}_{\{u = Mu\}}(t, x - \xi) m(t, x- \xi) + \\
& + \frac{\alpha}{\epsilon} \mathbb{1}_{\{u = Mu\}} m= 0 \text{ in } (0,T) \times \mathbb{T}^d; \end{aligned}\\
m(0) = m_0 \text{ in } \mathbb{T}^d;\\
(u -Mu) + (-\partial_t u -  \nu\Delta u - f(m)) \ne 0 \Rightarrow \alpha = 1; \\
\forall \xi \in K, u(t,x) \ne k(t,x,\xi) + u(t,x + \xi) \Rightarrow V(t,x,\xi) = 0 ;\\
0 \leq V(t,x,\xi) \leq 1.
\end{cases}
\end{equation*}
As we want to apply Kakutani's fixed point theorem on $\mathcal{F} := \mathcal{F}_2 \circ \mathcal{F}_1$, we need to prove that $\mathcal{F}$ is upper semicontinuous and that it is valued in the set of convex and closed subsets of $L^2((0,T), H^1(\mathbb{T}^d))$. As the last point is trivial, we focus in this proof on the upper semicontinuity. We recall that a set valued function $F$ from $A$ to $B$ is upper semicontinuous if for any open set $\mathcal{O} \subset B$, $\{x \in A, F(x) \subset \mathcal{O}\}$ is open in $A$. Let us take an open set $\mathcal{O} \subset L^2$ and $m \in L^2((0,T), H^1(\mathbb{T}^d))$ such that $\mathcal{F}(m) \subset \mathcal{O}$. In view of the previous part (namely lemma \ref{estimate2}), we can affirm that $\mathcal{F}(m)$ is a compact subset of $L^2((0,T),H^1( \mathbb{T}^d))$ endowed with its weak topology, hence 
\[
d := dist(\mathcal{F}(m), \mathcal{O}^c) > 0.
\]
Let us take $\delta > 0$ and $m' \in L^2((0,T), H^1(\mathbb{T}^d))$ such that $||m - m'||_{L^2(H^1)} \leq \delta$. We now prove that if $\delta$ is small enough, then $\mathcal{F}(m') \subset \mathcal{O}$, and thus that $\mathcal{F}$ is upper semi continuous. First, we define $u_1$ and $u_2$ by
\[
u_1 = \mathcal{F}_1(m) ; 
u_2 = \mathcal{F}_1(m').
\]
For any $µ_2 \in \mathcal{F}(m')$, there exists $(\alpha_2, V_2)$ such that $(µ_2, \alpha_2, V_2)$ solves
\[
\begin{cases}
\begin{aligned}
\partial_t µ_2 &-  \nu\Delta µ_2  - \sum_{\xi \in K} V_2(t,x- \xi,\xi) \frac{\alpha_2(t, x - \xi)}{\epsilon} \mathbb{1}_{\{u_2 = Mu_2\}}(t, x - \xi) µ_2(t, x- \xi) + \\
&+ \frac{\alpha_2}{\epsilon} \mathbb{1}_{\{u_2 = Mu_2\}} µ_2= 0 \text{ in } (0,T) \times \mathbb{T}^d; \end{aligned}\\
µ_2(0) = m_0 \text{ in } \mathbb{T}^d.\\
\end{cases}
\]
Now we define $\alpha_1$ on $\{u_1 = Mu_1\} \cap \{ - \partial_t u_1 -  \nu \Delta u_1 = f(m)\}$ by

\begin{itemize}
\item $\alpha_1 = \alpha_2 $ on $ \{u_2 = Mu_2 \} \cap \{ - \partial_t u_2 -   \nu \Delta u_2 = f(m')\} $;
\item $\alpha _1 = 1$ on $\{-\partial_t u_2 -  \nu \Delta u_2 < f(m')\}$;
\item $\alpha_1 = 0 $ on $\{ u_2 < Mu_2\}$;
\end{itemize}
and we set $\alpha_1 = 1$ elsewhere. Now we define $V_1$ by
\begin{itemize}
\item $\forall (t,x,\xi)$ such that $u_i (t,x,\xi) = k(t,x,\xi) + u_i(t,x,x + \xi)$, for $i \in \{1 ; 2\}$, then 
\[
V_1(t,x,\xi) = V_2(t,x,\xi);
\]
\item elsewhere the value of $V_1$ does not matter so we just define it in such a way that it satisfies 
\[
\begin{cases}
\forall \xi \in K, u(t,x) \ne k(t,x,\xi) + u(t,x,x + \xi) \Rightarrow V(t,x,\xi) = 0 ;\\
0 \leq V(t,x,\xi) \leq 1;
\end{cases}
\]
which is aways possible.
\end{itemize}
Let us define $µ_1$ as the unique solution of 
\[
\begin{cases}
\begin{aligned}
\partial_t µ_1 &- \nu \Delta µ_1  - \sum_{\xi \in K} V_1(t,x- \xi,\xi) \frac{\alpha_1(t, x - \xi)}{\epsilon} \mathbb{1}_{\{u_1 = Mu_1\}}(t, x - \xi) µ_1(t, x- \xi) + \\
& + \frac{\alpha_1}{\epsilon} \mathbb{1}_{\{u_1 = Mu_1\}} µ_1= 0 \text{ in } (0,T) \times \mathbb{T}^d;\end{aligned}\\
µ_1(0) = m_0 \text{ in } \mathbb{T}^d.
\end{cases}
\]
By construction, $µ_1 \in \mathcal{F}(m)$. We then define $µ := µ_1 - µ_2$. Once again by construction, $µ(0) = 0$ and $µ$ satisfies 
\[
\begin{aligned}
&\partial_t µ -  \nu \Delta µ + \frac{\alpha_2}{\epsilon} \mathbb{1}_A(\sum_{\xi \in K} V_2)µ  - \sum_{\xi \in K} V_2(t,x- \xi,\xi) \frac{\alpha_2(t, x - \xi)}{\epsilon} \mathbb{1}_{A}(t, x - \xi) µ(t, x- \xi)\\
& = \frac{\alpha_2}{\epsilon} \mathbb{1}_{\{u_2 = Mu_2\} \cap \{ u_1 < Mu_1\}} (\sum_{\xi \in K}V_2)µ_2\\
& - \sum_{\xi \in K} V_2(t,x- \xi,\xi) \frac{\alpha_2(t, x - \xi)}{\epsilon} \mathbb{1}_{\{u_2 = Mu_2\} \cap \{ u_1 < Mu_1\}}(t, x - \xi) µ_2(t, x- \xi)\\
& + \frac{\alpha_2}{\epsilon} \mathbb{1}_{\{u_2 = Mu_2\} \cap \{ u_1 = Mu_1\} \cap A^c} (\sum_{\xi \in K}V_2)µ_2\\
& - \sum_{\xi \in K} V_2(t,x- \xi,\xi) \frac{\alpha_2(t, x - \xi)}{\epsilon} \mathbb{1}_{\{u_2 = Mu_2\} \cap \{ u_1 = Mu_1\}\cap A ^c}(t, x - \xi) µ_2(t, x- \xi)\\
& - \frac{\alpha_1}{\epsilon} \mathbb{1}_{\{u_2 = Mu_2\} \cap \{ u_1 = Mu_1\} \cap A^c} (\sum_{\xi \in K}V_1)µ_1\\
& + \sum_{\xi \in K} V_1(t,x- \xi,\xi) \frac{\alpha_1(t, x - \xi)}{\epsilon} \mathbb{1}_{\{u_2 = Mu_2\} \cap \{ u_1 = Mu_1\} \cap A^c}(t, x - \xi) µ_1(t, x- \xi)\\
&- \frac{1}{\epsilon} \mathbb{1}_{\{P(u_2) = f(m')\} \cap \{ P(u_1) < f(m)\} \cap A^c} (\sum_{\xi \in K}V_1)µ_1\\
& + \sum_{\xi \in K} V_1(t,x- \xi,\xi) \frac{1}{\epsilon} \mathbb{1}_{\{P(u_2) = f(m')\} \cap \{ P(u_1) < f(m)\} \cap A^c}(t, x - \xi) µ_1(t, x- \xi).
\end{aligned}
\]
Where we have used for $A$ the set of coincidence :
\[
A := \{(t,x)/ \forall \xi,V_2(t,x,\xi) \alpha_2(t,x,\xi) \mathbb{1}_{\{u_2 = Mu_2\}}(t,x) = V_1(t,x,\xi) \alpha_1(t,x,\xi) \mathbb{1}_{\{u_1 = Mu_1\}}(t,x)  \}.
\]
The operator $P$ is here defined by being the parabolic operator :
\[
P(u) := -\partial_t u -  \nu\Delta u.
\]
Let us remark that all the terms in the right hand side of this expression are the ones which involve different coefficients (i.e. different $\alpha$ and $V$) in front of $µ_1$ and $µ_2$. We can also note that all the terms of the right hand side are multiplied by a characteristic function of a subsets of which the Lebesgue measure goes to $0$ as $||m' - m||_{L^2(H^1)}$ goes to $0$. Indeed, because $\mathcal{F}_1$ is continuous, 

\[
||u_2 - u_1||_{L^2((0,T),H^1(\mathbb{T}^d))} \rightarrow_{\delta \rightarrow 0} 0.
\]
Hence, taking $\delta$ sufficiently small, we obtain that the Lebesgue measure of the following sets are as small as we want :

\begin{itemize}
\item for all $\xi_1 \ne \xi_2 \in K$ : 
\[
\cap_{i \ne j}\{ u_i(t,x) = k(t,x,\xi_i) + u_i(t,x + \xi_i)\} \cap \{ u_j(t,x) < k(t,x,\xi_i) + u_j(t,x + \xi_i)\}
\]
\item $\{P(u_1) < f(m)\} \cap \{ P(u_2) = f(m')\}$
\item $\{u_1 < Mu_1\} \cap \{ u_2 = Mu_2\}$.
\end{itemize}
Thus because $µ_1$ and $µ_2$ are bounded in $L^2((0,T) \times \mathbb{T}^d)$ independently of $\delta$ (lemma \ref{estimate2}), we deduce that taking $\delta$ sufficiently small, the right hand side of the previous equation is as small as necessary in $L^2((0,T) \times \mathbb{T}^d)$. Thus, we fix $\eta > 0$ and we choose $\delta$ such that the right hand side is smaller than $\eta$ in the $L^2((0,T) \times \mathbb{T}^d)$ norm. Multiplying by $µ$ and integrating over $\mathbb{T}^d$ the equation $µ$ satisfies, we deduce that
\[
\frac{1}{2}\frac{d}{dt} ||µ(t)||^2_{L^2(\mathbb{T}^d)} \leq \frac{1}{\epsilon} \#(K) ||V ||_{L^{\infty}} ||µ||_{L^2(\mathbb{T}^d)}^2 + \eta ||µ||_{L^2(\mathbb{T}^d)};
\]
where $\#(K)$ stands for the cardinal of the set $K$. From this inequality, it follows that 
\[
\frac{d}{dt} ||µ(t)||_{L^2(\mathbb{T}^d)} \leq \frac{1}{\epsilon} \#(K) ||V ||_{L^{\infty}} ||µ||_{L^2(\mathbb{T}^d)} + \eta.
\]
Hence, we conclude with Gronwall's lemma that taking $\eta$ small enough, $µ$ is as small as necessary in $L^{\infty}((0,T), L^2(\mathbb{T}^d))$ (we recall that $µ(0) = 0$). Because of the partial differential equations satisfied by $µ$, it follows that taking $\delta$ small enough :
\[
||µ||_{L^2((0,T), H^1(\mathbb{T}^d))} < \frac{d}{2};
\]
which proves that $µ_2 \in \mathcal{O}$ and thus that $\mathcal{F}(m') \subset \mathcal{O}$. Hence $\mathcal{F}$ is upper semi continuous and we deduce from Kakutani's fixed point theorem the existence of a solution.
\end{proof}

\section{Existence of solutions of the MFG system}
In this section we present the existence of solutions of the MFG system $(\ref{MFGsystem})$. The proof of this result consists in passing to the limit in the penalized system. Let us remark that in the first part of this article, we use either hypothesis  \ref{hyp1} or \ref{hypmultiple} to pass to the limit $\epsilon \to 0$ in a Fokker-Planck equation of jumping particles. Here such an assumption is no more required, as the jumps the players are using, are by definition optimal for a certain optimization problem. Obviously this problem is the optimization problem the players have to solve. 

\begin{Theorem}
There exists a solution $(u,m) \in L^2((0,T), H^2(\mathbb{T}^d)) \cap H^1((0,T), L^2(\mathbb{T}^d)) \times L^2((0,T), H^1(\mathbb{T}^d))$ of $(\ref{MFGsystem})$. 
\end{Theorem}

\begin{proof}
For $\epsilon > 0$ we denote by $(u_{\epsilon}, m_{\epsilon}, \alpha_{\epsilon},V_{\epsilon})$ a solution of the penalized system $(\ref{penalizedMFG})$. We first show some compactness for the sequence $(m_{\epsilon})_{\epsilon > 0}$. Let us remark that 
\[
\begin{aligned}
D(m_{\epsilon}) &= \int_0^T \int_{\mathbb{T}^d} (-\partial_t u_{\epsilon} -  \nu\Delta u_{\epsilon})m_{\epsilon} - \int_{\mathbb{T}^d}u_{\epsilon}(0)m_0.%;\\
%& \leq \int_0^T \int_{\mathbb{T}^d}  f(m_{\epsilon})m_{\epsilon} - \int_{\mathbb{T}^d}u_{\epsilon}(0) m_0.
\end{aligned}
\]
Furthermore, because of lemma $\ref{estimate2}$,
\[
- D(m_{\epsilon}) + C_{\epsilon} ||m_0||_{L^2} ||m_{\epsilon}||_{L^2((0,T), H^1)} \geq ||m_{\epsilon}||_{L^2((0,T), H^1)}^2;
\]
where $C_{\epsilon}$ only depends on $||u_{\epsilon}||_{L^{\infty}}$ (and on $k$ which is fixed here). We then deduce from the assumptions we made on $f$, that there exists $C$ independent of $\epsilon$ such that :
\[
\begin{aligned}
||m_{\epsilon}||_{L^2((0,T), H^1)}^2 &\leq \int_0^T \int_{\mathbb{T}^d} (-\partial_t u_{\epsilon} -  \nu\Delta u_{\epsilon})m_{\epsilon} - \int_{\mathbb{T}^d}u_{\epsilon}(0)m_0 + C ||m_0||_{L^2} ||m_{\epsilon}||_{L^2((0,T), H^1)};\\
& \leq \int_0^T \int_{\mathbb{T}^d} f(m_{\epsilon})m_{\epsilon} + C ||m_0||_{L^2} ||m_{\epsilon}||_{L^2((0,T), H^1)}.
\end{aligned}
\]
Thus we deduce from the assumptions we made on $f$ that $(m_{\epsilon})_{\epsilon > 0}$ is a bounded sequence of $L^2((0,T), H^1(\mathbb{T}^d))$. So there exists $m \in L^2((0,T), H^1(\mathbb{T}^d))$ such that, extracting a subsequence if necessary, $(m_{\epsilon})_{\epsilon}$ weakly converges toward $m$ in $ L^2((0,T), H^1(\mathbb{T}^d))$. Because $f$ is continuous for the weak topology, we deduce from lemma \ref{continuityqvi} (see appendix), that $(u_{\epsilon})_{\epsilon}$ converges toward $u$ solution of the quasi variational inequality associated to 
\[
\begin{cases}
\max(-\partial_t u -  \nu \Delta u - f(m), u - Mu) = 0 \text{ in } (0,T) \times \mathbb{T}^d;\\
u(T) = 0 \text{ in } \mathbb{T}^d.
\end{cases}
\]
Moreover, by passing to the limit in 
\begin{equation}\label{optmmfg}
\begin{aligned}
&\forall v \in H, v \leq Mv, v(T) = 0 :\\
& \begin{aligned}\int_0^T(-\partial_t v -  \nu \Delta v, m_{\epsilon})_{H^{-1} \times H^1} -\int_{\mathbb{T}^d}v(0)m_0 &\geq \int_0^T(-\partial_t u_{\epsilon} -  \nu\Delta u_{\epsilon}, m_{\epsilon})_{H^{-1} \times H^1} - \int_{\mathbb{T}^d}u_{\epsilon}(0)m_0\\
&= D(m_{\epsilon});\end{aligned}
\end{aligned}
\end{equation}
we obtain
\[
\begin{cases}
D(m) > - \infty ; \\
\begin{aligned}
&\forall v, v \leq Mv, v(T) = 0 : \\
&\begin{aligned}\int_0^T(-\partial_t v - \nu \Delta v, m)_{H^{-1} \times H^1} - \int_{\mathbb{T}^d}v(0)m_0 &\geq \int_0^T(-\partial_t u - \nu \Delta u, m)_{H^{-1} \times H^1} - \int_{\mathbb{T}^d}u(0)m_0\\&= D(m).\end{aligned}
\end{aligned}
\end{cases}
\]
Let us note that we can pass to the limit in the right hand side of (\ref{optmmfg}) because the uniform bounds on $f$ yields some uniform H\"older estimates in time for $u_{\epsilon}$. Thus, we can easily deduce that extracting a subsequence if necessary :
\[
u_{\epsilon}(0) \underset{\epsilon \to 0}{\overset{L^2}{\longrightarrow}} u(0);
\]
\[
\int_0^T\int_{\mathbb{T}^d}m_{\epsilon}\partial_t u_{\epsilon} \underset{\epsilon \to 0}{\longrightarrow} \int_0^T\int_{\mathbb{T}^d}(\partial_t u,m)_{H^{-1}\times H^1}.
\]
Now let us remark that for all $\epsilon > 0$,
\[
\begin{aligned}
 - D(m_{\epsilon}) &= \frac{1}{\epsilon}\int_0^T \int_{\mathbb{T}^d} (\sum_{\xi \in K} V(t,x,\xi) k(x,\xi))\alpha_{\epsilon}(t,x) m_{\epsilon}(t,x)dx dt ;\\
 & \geq \frac{1}{\epsilon}\int_0^T \int_{\mathbb{T}^d} (\sum_{\xi \in K} V(t,x,\xi)) k_0 m_{\epsilon}(t,x) \mathbb{1}_{\{-\partial_t u_{\epsilon} -  \nu\Delta u_{\epsilon} < f(m_{\epsilon})\}}(t,x) dx dt; \\
 & \geq \frac{k_0}{\epsilon}\int_0^T \int_{\mathbb{T}^d} m_{\epsilon} \mathbb{1}_{\{-\partial_t u_{\epsilon} -  \nu\Delta u_{\epsilon} < f(m_{\epsilon})\}}.
\end{aligned}
\]
Since $(D(m_{\epsilon}))_{\epsilon > 0}$ is a bounded sequence (c.f. (\ref{optmmfg})), we deduce that 
\[
\big{(} \frac{k_0}{\epsilon}\int_0^T \int_{\mathbb{T}^d} m_{\epsilon} \mathbb{1}_{\{-\partial_t u_{\epsilon} -  \nu\Delta u_{\epsilon} < f(m_{\epsilon})\}}\big{)}_{\epsilon > 0}
\]
is also a bounded sequence and thus that :
\[
\int_0^T \int_{\mathbb{T}^d} (-\partial_t u -  \nu\Delta u - f(m))m = 0.
\]
This ends the proof of the fact that $(u,m)$ is a solution of $(\ref{MFGsystem})$.

\end{proof}

\section{Uniqueness of solutions of the MFG system}
We now turn to the question of the uniqueness of solutions of $(\ref{MFGsystem})$. As it is the case in MFG of continuous control \citep{lasry2007mean}, uniqueness does not hold in general. However it does under an assumption on the monotonicity of the costs of the MFG (i.e. the coupling) with respect to the density of players. In our model the density of players appears only in the running cost $f$ and thus only an assumption on $f$ is required for uniqueness to hold. We recall that $f$ is said to be strictly monotone if :
\[
\int_0^T \int_{\mathbb{T}^d} (f(m_1) - f(m_2))(m_1 - m_2) > 0 \text{ if } m_1 \ne m_2.
\]

\begin{Theorem}
Assume that $f$ is strictly monotone, then there exists at most one solution of $(\ref{MFGsystem})$.
\end{Theorem}
\begin{proof}
The proof of this statement is directly inspired from the original proof of uniqueness in MFG of continuous control \citep{lasry2007mean}. Let us take $(u_1, m_1)$ and $(u_2, m_2)$ two solutions of $(\ref{MFGsystem})$. We denote by $u$ and $m$ the differences $u_1 - u_2$ and $m_1 - m_2$. Let us observe that because of the optimality of $u_1$ in $D(m_1)$ and similarly for $u_2$ in $D(m_2)$ we obtain
\[
\begin{aligned}
\int_0^T \int_{\mathbb{T}^d} (-\partial_t u -  \nu\Delta u)m = \int_0^T& \int_{\mathbb{T}^d} (-\partial_t(u_1 - u_2) -  \nu\Delta (u_1 - u_2))m_1 \\
&+ \int_0^T \int_{\mathbb{T}^d} (-\partial_t (u_2 - u_1) -  \nu\Delta (u_2 - u_1))m_2; \\
 \leq 0&.
\end{aligned}
\]
On the other hand, because $m_2 \geq 0$, and 
\[
\int_0^T \int_{\mathbb{T}^d} (-\partial_t u_1 -  \nu\Delta u_1 - f(m_1))m_1 = 0;
\]
we deduce that
\[
\begin{aligned}
\int_0^T \int_{\mathbb{T}^d} (-\partial_t u_1 -  \nu\Delta u_1) m &= \int_0^T \int_{\mathbb{T}^d} f(m_1)m_1 + m_2(\partial_t u_1 +  \nu\Delta u_1);\\
& \geq \int_0^T \int_{\mathbb{T}^d} f(m_1)(m_1 - m_2).
\end{aligned}
\]
Where we have used the fact that $-\partial_t u_1 -  \nu\Delta u_1 \leq f(m_1)$. Obviously we have the analogous relation for $u_2$. Putting the pieces together we finally obtain 

\[
\int_0^T \int_{\mathbb{T}^d} (f(m_1) - f(m_2))(m_1 - m_2) \leq \int_0^T \int_{\mathbb{T}^d} (-\partial_t u - \nu \Delta u)m \leq 0.
\]
Using the strict monotonicity of $f$, we have just proven that $m_1 = m_2$ and thus that there exists at most one solution of $(\ref{MFGsystem})$.

\end{proof}

\section{The stationary setting}
In this section, we present a stationary setting for a MFG of impulse control. We denote by $k$ the cost of jumps and we assume that it satisfies (\ref{hypk}). We denote by $f$ the running cost for the players. We assume that 
\begin{itemize}
\item $f$ is continuous from $H^1(\mathbb{T}^d)$ endowed with its the weak topology to $H^{-1}(\mathbb{T}^d)$.
\item $f$ is uniformly bounded by below by a constant $-C$ (where $C >0$) on the positive elements of $H^1(\mathbb{T}^d)$.
\item $f$ maps $H^1(\mathbb{T}^d)$ into a bounded subset of $L^d(\mathbb{T}^d)$.
\end{itemize}
We denote by $\delta > 0$ the death rate of the players and by $\lambda > 0$ their intertemporal preference rate. We denote by $\rho \in L^2(\mathbb{T}^d)$, $\rho \geq 0$ the entry rate of the players. The jump operator $M$ is defined by :
\[
Mu(x) = \inf_{\xi \in K} k(x,\xi) + u(x + \xi).
\]
We are interested in the following MFG system :
\begin{equation}\label{statMFG}
\begin{cases}
\max(-\nu \Delta u + \lambda u - f(m), u - Mu) = 0 \text{ in } \mathbb{T}^d;\\
\forall v \in H^1(\mathbb{T}^d), v \leq Mv :\\
\nu \int_{\mathbb{T}^d}\nabla m \cdot \nabla (v-u) + \delta \int_{\mathbb{T}^d}m(v-u) \geq \int_{\mathbb{T}^d}\rho(v-u);\\
\int_{\mathbb{T}^d}(-\nu \Delta u + \lambda u - f(m))m = 0.
\end{cases}
\end{equation}
The following result holds true :
\begin{Theorem}\label{existuniquestat}
There exists a solution $(u,m) \in H^2(\mathbb{T}^d) \times H^1(\mathbb{T}^d)$ of (\ref{statMFG}). It is unique under the assumption that $f$ is strictly monotone.
\end{Theorem}
We do not detail the proof of this result as its argument follows step by step the one of the time dependent case, namely by passing to the limit in the following penalized system :
\[
\begin{cases}
\max(-\nu \Delta u + \lambda u - f(m), u - Mu) = 0 \text{ in } \mathbb{T}^d;\\
-\nu \Delta m + \delta m + \frac{\alpha}{\epsilon} \mathbb{1}_{\{u = Mu\}} m - \sum_{\xi \in K} V(x- \xi,\xi) \frac{\alpha( x - \xi)}{\epsilon} \mathbb{1}_{\{u = Mu\}}(x - \xi) m( x- \xi)= \rho \text{ in } \mathbb{T}^d ;\\
(u -Mu) + (- \nu \Delta u + \lambda u - f(m)) \ne 0 \Rightarrow \alpha = 1 ;\\
\forall \xi \in K, u(x) \ne k(x,\xi) + u(x + \xi) \Rightarrow V(x,\xi) = 0 ; \\
\forall \xi \in K, \forall (x) \in (0,T) \times \mathbb{T}^d, 0 \leq V(x,\xi) \leq 1;\\
 0 \leq \alpha \leq 1;
\end{cases}
\]
using the estimate (\ref{eststat}).
\section{The optimal control interpretation}
In this section, we present an optimal control interpretation of a MFG of impulse control. We do not make use of this interpretation to prove the existence of a solution of the MFG. We just show that a certain optimization problem has a solution, and that the solution of the MFG, for which we have proven the existence in the previous part, is the solution of this optimization problem. To make this section simpler we work only in the stationary setting. The case of the time dependent problem is mentioned at the end of this section. Let us note that the optimal control interpretation of MFG has been introduced in \citep{lasry2007mean} and that it can be used to show the existence of solutions for certain MFG system, see for instance \citep{cardaliaguet2015second}.\\
\\
We denote by $\mathcal{F}$ a strictly convex function from $H^1(\mathbb{T}^d)$ to $\mathbb{R}$, bounded by below. We assume that there exists $f$, satisfying the requirements of the previous section, such that for any $m, m' \in H^1(\mathbb{T}^d)$ :
$$\lim_{\theta \to 0} \frac{\mathcal{F}((1-\theta)m + \theta m') - \mathcal{F}(m)}{\theta} = \int_{\mathbb{T}^d} f(m)(m' - m).$$
 We define the following application from $H^1(\mathbb{T}^d)$ to $\mathbb{R}_-\cup \{- \infty\}$ :
\[
D(m) = \inf \{ \nu \int_{\mathbb{T}^d}\nabla m \cdot \nabla v + \delta \int_{\mathbb{T}^d}m v - \int_{\mathbb{T}^d}\rho v | v \in H^1(\mathbb{T}^d), v \leq Mv\}.
\]
The optimal control interpretation of the MFG of impulse control leads to the following optimization problem :
\begin{equation}\label{optcontr}
\inf_{m \in \mathcal{H}} \mathcal{F}(m) - D(m)
\end{equation}
where $\mathcal{H} =\{ m \in  H^1(\mathbb{T}^d), m \geq 0\}$. We now establish the following result :
\begin{Prop}\label{existmini}
Under the previous assumptions on $\mathcal{F}$, the problem (\ref{optcontr}) admits a unique minimizer $m^* \in \mathcal{H}$.
\end{Prop}
\begin{proof}
We first prove that the function $D(\cdot)$ is concave and upper semi continuous on $\mathcal{H}$. Assume that the sequence $(m_n)_{n \geq 0}$ converges weakly toward $m$ in $\mathcal{H}$. Let us remark that because $\mathcal{F}$ is bounded by below : $D(m) > - \infty$. Moreover for any $v \in H^1(\mathbb{T}^d), v \leq Mv$ we obtain
\[
\nu \int_{\mathbb{T}^d}\nabla (m_n -m)\cdot \nabla v + \delta \int_{\mathbb{T}^d}(m_n-m) v  \rightarrow_{n \to \infty} 0.
\]
Thus taking a sequence $(v_p)_{p\geq 0}$ in $H^1(\mathbb{T}^d)$ such that
\[
\nu \int_{\mathbb{T}^d}\nabla m \cdot \nabla v_p + \delta \int_{\mathbb{T}^d}m v_p - \int_{\mathbb{T}^d}\rho v_p \leq D(m) + \frac{1}{p+1};
\]
we deduce that
\[
\begin{aligned}
D(m_n) \leq \nu \int_{\mathbb{T}^d}\nabla m_n &\cdot \nabla v_p + \delta \int_{\mathbb{T}^d}m_n v_p - \int_{\mathbb{T}^d}\rho v_p;\\
\to_{n \to \infty} \nu \int_{\mathbb{T}^d}&\nabla m \cdot \nabla v_p + \delta \int_{\mathbb{T}^d}m v_p - \int_{\mathbb{T}^d}\rho v_p;\\
 \leq D&(m) + \frac{1}{p}.
\end{aligned}
\]
Thus $D(\cdot)$ is upper semi continuous, it is obviously concave. Now let $(m_n)_{n \geq 0}$ be a minimizing sequence of (\ref{optcontr}). If we denote by $µ$ the solution of 
\[
- \nu \Delta µ + \delta µ = \rho;
\]
then we observe that $µ \in \mathcal{H}$ and $D(µ) = 0$. Thus we deduce that
\[
0 \leq - D(m_n) \leq \mathcal{F}(µ) - \inf \mathcal{F}.
\]
Recalling lemma \ref{eststat}, $(m_n)_{n\geq 0}$ is thus a bounded sequence of $H^1(\mathbb{T}^d)$. Thus it converges weakly to $m^* \in H^1(\mathbb{T}^d)$. Because $\mathcal{F}$ is weakly sequentially lower semi continuous (it is continuous and convex) and $D(\cdot)$ is weakly sequentially upper semi continuous, we deduce that $m^*$ is a minimizer of (\ref{optcontr}). This minimizer is unique because $\mathcal{F}$ is strictly convex and $D$ is concave.
\end{proof}
Now let us remark that the solution $(u,m)$ of the MFG system (\ref{statMFG}) (with $\lambda = \delta$) given by theorem \ref{existuniquestat} satisfies the following system of variational inequalities:
\begin{equation}\label{systvi}
\begin{cases}
\forall µ \in H^1(\mathbb{T}^d), µ \geq 0 :\\
\int_{\mathbb{T}^d} (-\nu \Delta u + \delta u - f(m))(µ - m) \leq 0;\\
\forall v \in H^1(\mathbb{T}^d), v \leq Mv :\\
\nu \int_{\mathbb{T}^d}\nabla m \cdot \nabla (v-u) + \delta \int_{\mathbb{T}^d}m(v-u) \geq \int_{\mathbb{T}^d}\rho(v-u).
\end{cases}
\end{equation}
This system of variational inequalities is the characterization of a saddle point of (\ref{optcontr}). From this observation we deduce the following :
\begin{Theorem}
The unique minimizer of (\ref{optcontr}) is the density of players $m$ of the MFG of impulse control. This is if $(u,m)$ is the solution of (\ref{statMFG}) given by theorem \ref{existuniquestat}, then $m$ is the unique minimizer of (\ref{optcontr}).
\end{Theorem}
\begin{proof}
We denote by $(u,m)$ the unique solution of (\ref{statMFG}). For any $m' \in \mathcal{H}$, $0 <\theta < 1$, using the second variational inequality of (\ref{systvi}) we deduce that :
\[
\begin{aligned}
\mathcal{F}\left((1-\theta)m + \theta m'\right) - &D\left((1- \theta) m + \theta m'\right) - \mathcal{F}(m) + D(m)\\
 =\mathcal{F}\left((1-\theta)m + \theta m'\right) &- D\left((1- \theta) m + \theta m'\right) - \mathcal{F}(m)\\
& + \nu \int_{\mathbb{T}^d}\nabla m \cdot \nabla u + \delta \int_{\mathbb{T}^d}m u - \int_{\mathbb{T}^d}\rho u.
\end{aligned}
\]
Using the definition of $D(\cdot)$, we now obtain that :
\[
\begin{aligned}
\mathcal{F}\left((1-\theta)m + \theta m'\right) - D&\left((1- \theta) m + \theta m'\right) - \mathcal{F}(m) + D(m)\\
\geq \mathcal{F}\left((1-\theta)m + \theta m'\right) &- \mathcal{F}(m) - \nu \int_{\mathbb{T}^d}\nabla \left((1- \theta)m + \theta m' \right) \cdot \nabla u \\
- \delta \int_{\mathbb{T}^d}((1- \theta)m + &\theta m') u + \int_{\mathbb{T}^d}\rho u
+ \nu \int_{\mathbb{T}^d}\nabla m \cdot \nabla u + \delta \int_{\mathbb{T}^d}m u - \int_{\mathbb{T}^d}\rho u;\\
=  \mathcal{F}\left((1-\theta)m + \theta m'\right)& - \mathcal{F}(m) - \theta\int_{\mathbb{T}^d}(- \nu \Delta u + \delta u)(m' - m).
\end{aligned}
\]
Hence, using the first variational inequality of (\ref{systvi}), we obtain that :
\[
\lim_{\theta \to 0^+} \frac{\mathcal{F}\left((1-\theta)m + \theta m'\right) - D\left((1- \theta) m + \theta m'\right) - \mathcal{F}(m) + D(m)}{\theta} \geq 0.
\]
We deduce from the previous line that $m$ is a local minimum of the functional $\mathcal{F}(\cdot) - D(\cdot)$, because this functional is strictly convex, $m$ is the unique minimizer of (\ref{optcontr}).
\end{proof}

\subsection{The time dependent case}
We only indicate the optimal control interpretation in the time dependent case. We do not give any proofs are precise statements as they are the exact analogous of the one we just establish in the stationary setting. In a time dependent setting with time horizon $T> 0$ and initial condition $m_0 \in L^2(\mathbb{T}^d), m_0 \geq 0$, let us assume that there exists $\mathcal{F}$ such that for any $m, m' \in L^2((0,T), H^1(\mathbb{T}^d))$:
$$\lim_{\theta \to 0} \frac{\mathcal{F}((1-\theta)m + \theta m') - \mathcal{F}(m)}{\theta} = \int_0^T\int_{\mathbb{T}^d} f(m)(m' - m).$$
Then the optimal control problem associated with (\ref{MFGsystem}) is :
\[
\inf_{m \in \mathcal{H}} \mathcal{F}(m) - D(m).
\]
where $\mathcal{H} := \{ m \in L^2((0,T), H^1(\mathbb{T}^d)), m \geq 0\}$ and $D(m)$ is defined by (\ref{defDMFG}).

\section*{Acknowledgments} 
I would like to thank Pierre-Louis Lions (Coll\`ege de France) for his helpful advices on the general structure of the paper.\\
\\
This work is supported by a grant from the Fondation CFM pour la recherche.

\bibliographystyle{plainnat}
\bibliography{bib1}

\appendix
\section*{Appendix}

\section{Results on the impulse control problem}
The problem of impulse control is classical, we refer to the book of A. Bensoussan and J.-L. Lions \citep{bensoussan1984impulse} for a more complete presentation of the problem. The first part of this appendix is dedicated to time dependent quasi-variational inequalities(QVI), the second one to stationary QVI. We fix a probability space $(\Omega, \mathcal{A}, \mathbb{P})$.

\subsection{The time dependent setting}
In this time dependent setting, we fix a final time $T$. The problem of impulse control consists in minimizing the following expectation :
\begin{equation}\label{icpbm}
\inf_{(\tau_i)_i, (\xi_i)_i} \mathbb{E}[ \int_0^T f(s,X_s) ds + \sum_{ i =1 }^{\#(\tau_j)_j} k(\tau_i, X_{\tau_i^-}, \xi_i)];
\end{equation}
where the infimum is taken over the (finite and infinite) sequences $(\tau_i)_i$ of times such that $0 \leq \tau_i < \tau_{i+1}$ and over the sequences $(\xi_i)_i$ valued in the finite set $K$. The (random) sequences $(\tau_i)_i$ and $(\xi_i)_i$ are measurable with respect to the  $\sigma$-algebra generated by the process $(X_s)_{s\geq 0}$, which is defined below in (\ref{defX}). The function $f \in L^2((0,T), \mathbb{T}^d)$ denotes the running cost and $k \in L^{\infty}((0,T)  \times \mathbb{T}^d \times K)$ denotes the cost of the jumps (i.e. $k(t,x,\xi)$ is the cost paid to use the jump $\xi$ at the time $t$ and the position $x$). In $(\ref{icpbm})$, $(X_s)_s$ is the process given by
\begin{equation}\label{defX}
\begin{cases}
\forall s \in (\tau_i, \tau_{i+1}), dX_s = \sqrt{2 \nu}dW_s; \\
X_{\tau_i^+} = X_{\tau_i^-} + \xi_i ;\\
X_0 = x \in \mathbb{T}^d;
\end{cases}
\end{equation}
where $(W_s)_s$ is a standard brownian motion under  $(\Omega, \mathcal{A}, \mathbb{P})$. The problem of impulse control then consists in choosing the optimal jumps (defined by a time and an element of $K$) to impose on the state $(X_s)_{s \geq 0 }$ in order to minimize $(\ref{icpbm})$. We define $M(k,u)$ by
\[
M(k,u)(t,x) = \inf_{\xi \in K} \{ u(t,x + \xi) + k(t,x,\xi)\}.
\]
We shall note $Mu$ instead of $M(k,u)$ when there is no ambiguity on $k$. Several assumptions can be made on the regularity of $k$ as well as on its dependence on the variable of the problem. We here assume that the following holds in order to work with solutions of the problem which are smooth enough:
\begin{equation}\label{hypk2}
\begin{cases}
\forall \xi \in K, k( \cdot, \xi) \in H^2(\mathbb{T}^d); \\
k^* :x \to \inf_{\xi \in K} k(x, \xi) \in W^{2, \infty}(\mathbb{T}^d);\\
\exists k_0 > 0 \text{ such that } k \geq k_0.
\end{cases}
\end{equation}
We define $H$ by:
$$
H = \{ v \in L^2((0,T), H^1(\mathbb{T}^d)) , \partial_t v \in L^2((0,T), H^{-1}(\mathbb{T}^d))\}.
$$
In the same way value functions of optimal stopping problems can be solutions of obstacle problems \citep{bensoussan2011applications}, we expect the value function of this impulse control problem to be a solution of 
\begin{equation}\label{qvi}
\begin{cases}
\max(u- Mu, -\partial_t u -  \nu\Delta u -f) = 0 \text{ in } (0,T)\times \mathbb{T}^d;\\
u(T) = 0 \text{ in } \mathbb{T}^d.
\end{cases}
\end{equation}
However, just as variational inequalities are the most natural object to represent solutions of obstacle problems, quasi-variational inequalities (QVI) are a natural object associated to $(\ref{qvi})$. The QVI for this impulse control problem, which we denote by $QVI(f,k)$, is:
\begin{equation}{QVI(f,k)}\label{qvifk}
\begin{cases}
u \leq Mu \text{ a.e. in } (0,T)\times \mathbb{T}^d;\\
\forall v \in L^2((0,T),H^1( \mathbb{T}^d)), v \leq Mu, v(T) = 0;\\
-\int_0^T \int_{\mathbb{T}^d} \partial_t u (v -u) + \nu \int_0^T \int_{\mathbb{T}^d} \nabla u \cdot \nabla(v-u) \geq \int_0^T \int_{\mathbb{T}^d} f (v-u);\\
u(T) = 0 \text{ in } \mathbb{T}^d.
\end{cases}
\end{equation}
The function $u$ is here the solution/unknown of $QVI(f,k)$. Finding a solution of $QVI(f,k)$ is not possible for any $f \in L^2$ in any dimension. This is a consequence of the fact that if $f$ is not bounded by below, then we cannot expect in general $u$ to be bounded by below. Indeed in such a case, it is unclear in which sense the condition $u \leq Mu$ has to be understood. Usually a solution of (\ref{qvifk}) is build as the limit of the sequence $(u_n)_{n \in \mathbb{N}}$ defined by :
\begin{equation}\label{defun}
\begin{cases}
\max(-\partial_t u_{n+1} -  \nu \Delta u_{n+1} - f, u_{n+1} - Mu_n) = 0 \text{ in } (0,T) \times \mathbb{T}^d;\\
u_{n+1}(T) = 0 \text{ in } \mathbb{T}^d;
\end{cases}
\end{equation}
with the convention $u_{-1} = + \infty$. The obstacle problem (\ref{defun}) is understand in the sense of variational inequalities. For all $n \in \mathbb{N}$, $u_n \in L^2((0,T), H^2(\mathbb{T}^d))\cap H^1((0,T), L^2(\mathbb{T}^d))$ because (\ref{hypk2}) holds. Moreover, $(u_n)_{n \in \mathbb{N}}$ is a decreasing sequence because $M$ is monotone. If one can find $v \in L^2((0,T), H^1(\mathbb{T}^d))\cap H^1((0,T), H^{-1}(\mathbb{T}^d))$ such that $v \leq u_n$ for all $n \in \mathbb{N}$ then we deduce that :
\begin{equation*}
-\int_0^T \int_{\mathbb{T}^d}\partial_t u_n (v - u_n) + \nu \int_0^T \int_{\mathbb{T}^d} \nabla u_n \cdot \nabla(v - u_n) \geq \int_0^T\int_{\mathbb{T}^d}f (v - u_n).
\end{equation*}
Rearranging this inequality we deduce that :
\begin{equation}\label{estimateun}
\begin{aligned}
\sup_{0 \leq t\leq T} ||u_n(t)||_{L^2}^2 + \nu \int_0^T \int_{\mathbb{T}^d} |\nabla u_n|^2 \leq &\int_{\mathbb{T}^d} u_n(0) v(0) + \int_0^T \int_{\mathbb{T}^d} (\partial_t v -   \nu\Delta v + f, u_n)_{H^{-1} \times H^1}\\
& - \int_0^T \int_{\mathbb{T}^d}f v
\end{aligned}
\end{equation}
Thus we obtain estimates on the sequence $(u_n)_{n \in \mathbb{N}}$ from the existence of a uniform lower bound $v$. Let us note that if $f \geq 0$, then $u_n \geq 0$ for all $n \in \mathbb{N}$ so we can choose $v = 0$ in (\ref{estimateun}). More generally if $f$ is bounded by below by a constant $-C$ then for all $n \in \mathbb{N}$ we deduce that $u_n(t,x) \geq -Ct$ for all $(t,x) \in (0,T) \times \mathbb{T}^d$ and we can choose $v = - Ct$ in (\ref{estimateun}). Moreover, if $f$ is bounded by below by a constant $-C$ , because $k$ satisfies (\ref{hypk2}), the following estimate is classical : 
\begin{equation}\label{reg2qvi}
||\partial_t u||_{L^2} + ||u||_{L^2(H^2)} \leq C_1(1 + ||f||_{L^2})
\end{equation}
where $C_1$ depends only on $C$ and on $k$. We present a result of stability concerning solutions of regular QVI. This result does not seem to be new but we detail the proof for the sake of completeness.

\begin{Prop}\label{continuityqvi}
Let us take any sequence $(f_n)_n$ and a constant $C >0$, such that for all $n \in \mathbb{N}$, $f_n \in L^2((0,T) \times \mathbb{T}^d)$ and $f_n \geq -C$ . We also assume that $k \in L^{\infty}( \mathbb{T}^d \times K)$ satisfies $(\ref{hypk2})$. If $(f_n)_n$ is bounded in  $L^2((0,T) \times \mathbb{T}^d)$ and converges toward $f \in L^2((0,T) \times \mathbb{T}^d)$ in $L^2((0,T), H^{-1}(\mathbb{T}^d))$ with $f\geq - C$, then the sequence $(u_n)_n$ of solutions of $QVI(f_n, k)$ converges toward the solution $u$ of $QVI(f,k)$ in $L^2((0,T), H^1(\mathbb{T}^d))$. 
\end{Prop}

\begin{proof}
The sequence $(||f_n||_{L^2})_n$ is bounded. Hence $(u_n)_n$ is a bounded sequence of $L^2((0,T), H^2(\mathbb{T}^d))\cap H^1((0,T), L^2(\mathbb{T}^d))$. Extracting a subsequence if necessary, it converges to a limit $u^* \in L^2((0,T), H^2(\mathbb{T}^d))$ for the $L^2((0,T), H^1(\mathbb{T}^d))$ norm. The limit $u^*$ satisfies $u^* \leq Mu^*$ almost everywhere. Let us take $v \in L^2((0,T), H^1(\mathbb{T}^d))$ such that $v \leq Mu^*$. Obviously the following holds :
\[
v_n := v- Mu^* + M u_n \leq M u_n.
\]
Thus because of $QVI(f_n,k)$, we obtain
\[
-\int_0^T \int_{\mathbb{T}^d} \partial_t u_n (v_n -u_n) +  \nu\int_0^T \int_{\mathbb{T}^d} \nabla u_n \cdot \nabla(v_n-u_n) \geq \int_0^T \int_{\mathbb{T}^d} f_n (v_n-u_n).
\]
Re arranging this inequality leads to
\[
\begin{aligned}
\int_0^T \int_{\mathbb{T}^d} f_n (v-u_n) \leq & -\int_0^T \int_{\mathbb{T}^d} \partial_t u_n (v -u_n) +  \nu\int_0^T \int_{\mathbb{T}^d} \nabla u_n \cdot \nabla(v-u_n);\\
& - \int_0^T \int_{\mathbb{T}^d} (\partial_t u_n +  \nu\Delta u_n - f_n)(M u_n - Mu^*).
\end{aligned}
\]
Let us remark that $(||M u_n - Mu^*||_{L^2})_n$ converges to $0$ as $n$ goes to infinity. Thus, because $(\partial_t u_n + \Delta u_n - f_n)_n$ is bounded in $L^2$, passing to the limit in the previous equation we obtain
\[
-\int_0^T \int_{\mathbb{T}^d} \partial_t u^* (v -u^*) + \nu \int_0^T \int_{\mathbb{T}^d} \nabla u^* \cdot \nabla(v-u^*) \geq \int_0^T \int_{\mathbb{T}^d} f (v-u^*).
\]
We conclude by the uniqueness of solutions of $QVI$ \citep{laetsch1975uniqueness}, that $u^* = u$, the only solution of this QVI.
\end{proof}

We now present a result on weaker QVI. To pass to the limit $\epsilon \to 0$ in (\ref{onejump}), we need an estimate for right hand side $f$ which are only in $L^2((0,T), H^{-1}(\mathbb{T}^d))$. The following lemma gives such an estimate for a QVI in which we do not impose the constraint $u \leq M(k,u)$ on the whole space but only on the part which is of interest in $(\ref{onejump})$, for a given cost function $k$. The new constraint we impose is that 
\[
\forall \xi \in K : \mathbb{1}_{\{V(\xi) > 0 \}} (u(t,x) - k(x,\xi) - u(t,x + \xi)) \leq 0.
\]
We note $\mathcal{K}(k,u)$ the convex closed set :
\[
\mathcal{K}(k,u) := \{ v \in L^2((0,T), H^1(\mathbb{T}^d)), \forall \xi \in K : \mathbb{1}_{\{V(\xi) > 0 \}} (v(t,x) - k(x,\xi) - v(t,x + \xi)) \leq 0\}.
\]
\begin{Prop}\label{weakqvi}
Assume that there exists $V$ satisfying (\ref{hypV}) for which hypothesis \ref{hypmultiple} (in part 1 ) holds. We note $k$ and $w$ the couple given by hypothesis \ref{hypmultiple}. Then for any $f \in L^2((0,T), H^{-1}(\mathbb{T}^d)) \cap \mathcal{M}_b((0,T) \times \mathbb{T}^d)$ there exists $u \in L^2((0,T), H^1(\mathbb{T}^d))$ such that :
\begin{equation}\label{weakqviformulation}
\begin{cases}
u \in \mathcal{K}(k,u);\\
\forall v \in H, v \in \mathcal{K}(k,u);\\
-\int_0^T \int_{\mathbb{T}^d} \partial_t v (v -u) + \nu \int_0^T \int_{\mathbb{T}^d} \nabla u \cdot \nabla(v-u) + \frac{1}{2}\int_{\mathbb{T}^d}|v(T)|^2 \geq \int_0^T \int_{\mathbb{T}^d} f (v-u);
\end{cases}
\end{equation}
Moreover we have the estimate 
\[
||u||_{L^{\infty}(L^2)} + ||u||_{L^2(H^1)} \leq C ( 1 + ||f||_{L^2(H^{-1})});
\]
where $C$ only depends on $K$ and $\frac{||w||_{\infty}}{ \inf k}$.
\end{Prop}
The idea of the proof is that the QVI (\ref{weakqviformulation}) is associated to a formal impulse control problem in which one can only use the impulse control $\xi$ on $\{V(\xi) > 0 \}$. Because hypothesis \ref{hypmultiple} is satisfied, the QVI is somehow well defined and thus we can solve it for unbounded cost functions $f$.

\begin{proof}
Denoting $k$ and $w$ the functions given by hypothesis \ref{hypmultiple}, there exists $n^* \in \mathbb{N}$ such that :
\[
\forall (t,x) \in \cup_{\xi \in K}, \nexists (\xi_1, ..., \xi_p) \in K^p, p \geq n^*, \forall k \leq (p-1), (t,x + \sum_{i=1}^k \xi_i) \in \{V(\xi_p) > 0 \}.
\]
This fact is a direct consequence of $w \in L^{\infty}$ and is obtained by evaluating $w(t, x + \sum_{i =1}^k \xi_i)$. Moreover, we have :
\[
n^* \leq \frac{2 ||w||_{L^{\infty}}}{ \inf_{x, \xi} k(x, \xi)}
\]

We now define $\tilde{f}$ by
\[
\tilde{f} = \min_{p \leq n^*} \min_{(\xi_1, .. , \xi_p)} f( \cdot, \cdot + \sum_{i=1}^p \xi_i);
\]
The function $\tilde{f} \in L^2((0,T), H^{-1}(\mathbb{T}^d))$ is well defined because $f \in L^2((0,T), H^{-1}(\mathbb{T}^d)) \cap \mathcal{M}_b((0,T) \times \mathbb{T}^d)$ and it represents the best running cost one can face by jumping at the same time a maximum of $n^*$ times. We define $\tilde{u}$ by:
\[
\begin{cases}
-\partial_t \tilde{u}-  \nu\Delta \tilde{u} = \tilde{f} \text{ in } (0,T) \times \mathbb{T}^d;\\
\tilde{u}(T) = 0 \text{ in } \mathbb{T}^d.
\end{cases}
\]
As already mentioned above, an existence result for QVI usually comes from the existence of a lower bound for an approximating sequence. The function $\tilde{u}$ plays the role of a lower bound for the sequence $(u_n)_{n \in \mathbb{N}}$ that we now define. We denote by $u_0 \in H$ the only solution of
\[
\begin{cases}
-\partial_t u_0 -  \nu\Delta u_0 = f \text{ in } (0,T) \times \mathbb{T}^d;\\
u_0(T) = 0 \text{ in } \mathbb{T}^d.
\end{cases}
\]
We then define for all $n \in \mathbb{N}$, $u_n \in L^2((0,T), H^1(\mathbb{T}^d))$ by $u_{n+1} \in L^2((0,T), H^1(\mathbb{T}^d))$ is a solution of the weak variational inequality (we refer to \citep{bensoussan1984impulse} for a presentation of weak variational inequalities) :
\begin{equation}\label{defunqvi}
\begin{cases}
\forall v \in H, v \in \mathcal{K}(k,u_{n+1});\\
-\int_0^T \int_{\mathbb{T}^d} \partial_t v (v -u_{n+1}) +  \nu\int_0^T \int_{\mathbb{T}^d} \nabla u_{n+1} \cdot \nabla(v-u_{n+1}) + \frac{1}{2}\int_{\mathbb{T}^d} | v(T)|^2  \geq \int_0^T \int_{\mathbb{T}^d} f (v-u_{n+1}).
\end{cases}
\end{equation}
Straightforwardly, we deduce iteratively that for every $n \in \mathbb{N}$, $(u_n)_{n \in \mathbb{N}}$ is well defined, $u_{n+1} \leq u_n$, $\tilde{u} \in \mathcal{K}(k,u_n)$. The last point is a direct consequence of the definition of $\mathcal{K}(k, \cdot)$ and $\tilde{u}$. Evaluating the second line of (\ref{defunqvi}) with $v = \tilde{u}$, we deduce :
\begin{equation}\label{preestimate}
 \nu\int_0^T \int_{\mathbb{T}^d} |\nabla u_n|^2 \leq - \int_0^T (- \partial_t \tilde{u} - f, \tilde{u} - u_n)_{H^{-1} \times H^1} +  \nu\int_0^T \int_{\mathbb{T}^d}\nabla u_n \cdot \nabla \tilde{u} .
\end{equation}
Thus, $(u_n)_{n \in \mathbb{T}^d}$ is a bounded sequence of $L^2((0,T), H^1(\mathbb{T}^d))$. Because, it is also a decreasing sequence, it converges weakly in $L^2((0,T), H^1(\mathbb{T}^d))$ to a limit $u \in L^2((0,T), H^1(\mathbb{T}^d))$. It follows that
\[
\begin{cases}
\forall v \in H, v \in \mathcal{K}(k,u);\\
-\int_0^T \int_{\mathbb{T}^d} \partial_t v (v -u) +  \nu\int_0^T \int_{\mathbb{T}^d} \nabla u \cdot \nabla(v-u) + \frac{1}{2}\int_{\mathbb{T}^d} | v(T)|^2 \geq \int_0^T \int_{\mathbb{T}^d} f (v-u).
\end{cases}
\]
Moreover, because 
\[
\tilde{u} \leq u \leq u_0,
\]
and $\tilde{u} , u_0 \in L^{\infty}((0,T), L^2(\mathbb{T}^d))$, we obtain that $u \in L^{\infty}((0,T), L^2(\mathbb{T}^d))$. Finally, let us remark that 
\[
||\tilde{u}||_{L^{\infty}(L^2)} + ||\tilde{u}||_{L^2(H^1)} \leq C ( 1 + ||\tilde{f}||_{L^2(H^{-1})}),
\]
where $C$ does not depend on $f$. Moreover, by construction of $\tilde{f}$, 
\[
||\tilde{f}||_{L^2(H^{-1})} \leq \tilde{C} ||f||_{L^2(H^{-1})},
\]
where $\tilde{C}$ depends only on $K$, and $n^*$. Hence, $u$ satisfies :
\[
||u||_{L^{\infty}(L^2)} + ||u||_{L^2(H^1)} \leq C ( 1 + ||f||_{L^2(H^{-1})}),
\]
where $C$ depends only on $K$ and $\frac{||w||_{\infty}}{ \inf k}$ (and on $\nu$ and $d$).
\end{proof}

\subsection{The stationary setting}
In this section we give the analogue of the results of the previous section in a stationary setting. The two results in question are proved by following exactly the same argument as in the previous part. Thus we do not detail the proofs of those results. We still fix a finite set $K \subset \mathbb{T}^d$ and we define $M(k,u)$ by
\[
M(k,u)(x) = \inf_{\xi \in K} k(x,\xi) + u(x + \xi).
\]
We fix a parameter $\lambda > 0$ which describes the intertemporal preference rate in the following impulse control problem :
\begin{equation}\label{icp2}
\inf_{(\tau_i)_i, (\xi_i)_i} \mathbb{E}[ \int_0^{\infty} e^{-\lambda s}f(X_s) ds + \sum_{ i =1 }^{\#(\tau_j)_j} e^{-\lambda \tau_i}k(X_{\tau_i^-}, \xi_i)];
\end{equation}
where the trajectories are given by :
\begin{equation*}
\begin{cases}
\forall s \in (\tau_i, \tau_{i+1}), dX_s = \sqrt{2 \nu}dW_s; \\
X_{\tau_i^+} = X_{\tau_i^-} + \xi_i ;\\
X_0 = x \in \mathbb{T}^d;
\end{cases}
\end{equation*}
where $(W_s)_{s \geq 0}$ is a standard brownian motion under $(\Omega, \mathcal{A}, \mathbb{P})$ and $(\xi_i)_i$ and $(\tau_i)_i$ are the controls. In this problem $f$ is the running cost and $k$ the cost of jumps. We made for the rest of this section the assumption that $k$ satisfies (\ref{hypk2}). If the running cost $f \in L^2(\mathbb{T}^d)$ and $f \geq -C$ for some positive constant $C$, then the value function $u$ of (\ref{icp2}) is the unique solution in $H^1(\mathbb{T}^d)$ of the following QVI :
\begin{equation}{SQVI(f,k)}\label{sqvifk}
\begin{cases}
u \leq Mu \text{ a.e. in } \mathbb{T}^d;\\
\forall v \in H^1( \mathbb{T}^d), v \leq M(k,u);\\
  \nu \int_{\mathbb{T}^d} \nabla u \cdot \nabla(v-u) + \int_{\mathbb{T}^d} \lambda u(v -u) \geq  \int_{\mathbb{T}^d} f (v-u).
\end{cases}
\end{equation}
Moreover, $u$ is in fact in $H^2(\mathbb{T}^d)$ and satisfies in $L^2$ :
\[
\max(- \nu\Delta u + \lambda u - f, u - Mu) = 0 \text{ in } \mathbb{T}^d.
\]
We have the following result : 
\begin{Prop}\label{continuityqvi2}
Let us take any sequence $(f_n)_n$ and a constant $C >0$, such that for all $n \in \mathbb{N}$, $f_n \in L^2(\mathbb{T}^d)$ and $f_n \geq -C$ . We also assume that $k \in L^{\infty}( \mathbb{T}^d \times K)$ satisfies $(\ref{hypk2})$. If $(f_n)_n$ converges toward $f \in L^2( \mathbb{T}^d)$ in $L^2$ with $f\geq - C$, then the sequence $(u_n)_n$ of solutions of $SQVI(f_n, k)$ converges toward the solution $u$ of $SQVI(f,k)$ in $H^1(\mathbb{T}^d)$. 
\end{Prop}
We introduce the following notation :
We note $\mathcal{K}(k,u)$ the convex closed set :
\[
\mathcal{K}(k,u) := \{ v \in H^1(\mathbb{T}^d), \forall \xi \in K : \mathbb{1}_{\{V(\xi) > 0 \}} (v(x) - k(x,\xi) - v(x + \xi)) \leq 0\}.
\]

\begin{Prop}\label{weakqvi2}
Assume that $V$ satisfies hypothesis \ref{hypmultiple2} ( in part 1). We denote by $k$ and $w$ the couple given by hypothesis \ref{hypmultiple2}. Then for any $f \in H^{-1}(\mathbb{T}^d) \cap \mathcal{M}_b(\mathbb{T}^d)$ there exists $u \in H^1(\mathbb{T}^d)$ such that :
\begin{equation}\label{weakqviformulation2}
\begin{cases}
u \in \mathcal{K}(k,u);\\
\forall v \in \mathcal{K}(k,u);\\
 \nu  \int_{\mathbb{T}^d} \nabla u \cdot \nabla(v-u) + \int_{\mathbb{T}^d} \lambda u(v -u) \geq  \int_{\mathbb{T}^d} f (v-u);
\end{cases}
\end{equation}
Moreover we have the estimate 
\[
||u||_{L^2} + ||u||_{H^1} \leq C ( 1 + ||f||_{H^{-1}});
\]
where $C$ only depends on $K$, $\lambda$, $\nu$ and $\frac{||w||_{\infty}}{ \inf k}$.
\end{Prop}

\section{Some results on parabolic PDE in time dependent domains}
In this appendix, we are interested in two topics. The first one is the existence of solutions of parabolic PDE in a time dependent domain with non linear boundary conditions of Dirichlet type. The second one is the uniqueness of solutions of a similar problem with Dirichlet boundary conditions. The first results allow us to build functions needed in the proof of the uniqueness of solutions of the Fokker-Planck equations in the first part. The second one is a result which generalizes existing results on the well-posedness of a parabolic PDE in a time dependent domain as we allow the domain to evolve more generally than in the existing literature.

In this appendix, we focus on the heat operator
\[
\partial_t u - \nu \Delta u
\]
where $\nu > 0$, and the domain $B$ on which is posed the PDE is such that $B \subset [0,T] \times \mathbb{T}^d$, where $T >0$ is the final time. We denote by $A = [0,T] \times \mathbb{T}^d \setminus B$.

Given a function $\tilde{u} \in L^2(\mathbb{T}^d)$ and a running cost $f \in L^{\infty}((0,T) \times \mathbb{T}^d)$, the standard Dirichlet problem for a parabolic PDE is  :
\begin{equation}\label{tdpde}
\begin{cases}
\partial_t u - \nu \Delta u = f \text{ in } B;\\
u = \tilde{u} \text{ on } A \cup \{0\}\times \mathbb{T}^d.
\end{cases}
\end{equation}
We denote by $H := L^2((0,T), H^1(\mathbb{T}^d)) \cap H^1((0,T), H^{-1}(\mathbb{T}^d))$. Here we say that $u \in H$ is a strong solution of (\ref{tdpde}) if for any $v \in L^2((0,T),H^1(\mathbb{T}^d))$ such that $v = 0$ almost everywhere on $A$ :
\[
\begin{cases}
\int_0^T \int_{\mathbb{T}^d}(\partial_t u - \nu \Delta u - f, v)_{H^{-1}\times H^1} = 0;\\
u(0) = \tilde{u} \text{ a. e. on } A \cup \{0\}\times \mathbb{T}^d.
\end{cases}
\]
and we say that $u \in L^2((0,T), H^1(\mathbb{T}^d))$ is a weak solution of (\ref{tdpde}) if for any $v \in H$, such that $v = 0$ almost everywhere on $A$ and $v(T) = 0$ :
\[
\int_0^T \int_{\mathbb{T}^d}(-\partial_t v - \nu \Delta v - f, u)_{H^{-1}\times H^1} - \int_{\mathbb{T}^d}\tilde{u}(0)v(0)= 0.
\]
Regularity of solutions of (\ref{tdpde}) usually comes from two types of assumptions on $B$. The first class of assumptions consists in assuming that the set $B$ has a smooth boundary, namely that it evolves smoothly in time, see \citep{calvo2017parabolic}. The second class of assumptions is concerned with monotonicity assumption on the set $B$ see \citep{gianazza1996abstract}. In this last paper, the authors assume that $B$ is non decreasing in time (for the inclusion), that is $A$ is non-increasing in time (for the inclusion).

\subsection{Parabolic PDE in a time dependent domain with non linear boundary conditions}
Through this section, we assume that there exists $V, k, \tilde{u}$ such that
\[
\begin{cases}
V \in L^{\infty}(K,(0,T),\mathbb{T}^d) ;\\
V \geq 0 ;\\
\sum_{\xi \in K} V(\xi,t,x) \leq 1;
\end{cases}
\]
the function $k$ satisfies $(\ref{hypk})$ and $u^* \in H$ is such that
\begin{equation}\label{hypapp}
\begin{cases}
u^* \leq M(k,u^*) \text{ in } (0,T) \times \mathbb{T}^d;\\
u(T) = 0; \\
\forall \xi \in K, V(\xi,t,x) (k(x,\xi) + u^*(t, x+ \xi) - u^*(t,x)) = 0 \text{ a.e. in } (0,T) \times \mathbb{T}^d ;\\
\forall \xi \in K, x \in \tilde{A}_{\xi} \Rightarrow (k(x,\xi) + u^*(x+ \xi) - u^*(x)) = 0 ;\\
\forall \xi \in K, x \in \tilde{A}_{\xi} \Rightarrow \forall \xi' \ne \xi, u^*(x) < k(x,\xi') + u^*(x + \xi').
\end{cases}
\end{equation}

We are here interested in the existence of solutions of the following problem
\begin{equation}\label{qvitdpde}
\begin{cases}
\partial_t u - \nu \Delta u = f \text{ in } B;\\
\forall \xi \in K, V(\xi,t,x) (k(x,\xi) + u(t, x+ \xi) - u(t,x)) = 0 \text{ a.e. in } (0,T) \times \mathbb{T}^d ;\\
u(0) = u^*(0).
\end{cases}
\end{equation}
The problem (\ref{qvitdpde}) is of the type of (\ref{tdpde}) except for the fact that the function which gives the boundary conditions $\tilde{u}$ is here itself a function of $u$. The set $A$ is in this case given by $$A = \cup_{\xi \in K} \{ V(\xi,t,x) > 0\}.$$ We define $B$ by $$B = [0,T] \times \mathbb{T}^d \subset A.$$ The existence of $u^*$ allows us to state the following result :

\begin{Theorem}\label{existenceapp}
Let us assume that (\ref{hypapp}) holds true. If either $B$ is an open set with Lipschitz boundary such that $B \cap (\{0\} \times \mathbb{T}^d)$ is an open set with Lipschitz boundary, or $B$ is non-decreasing in time (for the inclusion), then there exists $\epsilon > 0$ such that for any $f \in L^{\infty}((0,T) \times \mathbb{T}^d)$ such that $||f||_{L^{\infty}} \leq \epsilon$ there exists a strong solution $u$ of :
\begin{equation*}
\begin{cases}
\partial_t u - \nu \Delta u = \partial_t u^* - \nu \Delta u^* + f \text{ in } B;\\
\forall \xi \in K, V(\xi,t,x) (k(x,\xi) + u(t, x+ \xi) - u(t,x)) = 0 \text{ a.e. in } (0,T) \times \mathbb{T}^d ;\\
u(0) = u^*(0).
\end{cases}
\end{equation*}
\end{Theorem}

\begin{Rem}
Because of the nature of the operator $M$, the result is resticted to small $f$. There are more regular $M$ for which this is not the case. For instance if $K = \{\xi\}$ is a singleton, then the previous result holds true for any $f\in L^{\infty}((0,T) \times \mathbb{T}^d)$.
\end{Rem}

\begin{proof}
We fix $f\in L^{\infty}((0,T)\times \mathbb{T}^d)$ and we define $\tilde{f} := \partial_t u^* - \nu \Delta u^* + f $ . In either one of the two cases (concerning the assumption on $B$), the following operator is well defined :
\[
\begin{aligned}
\mathcal{T} : &H \to H\\
&u \to \mathcal{T}(u)
\end{aligned}
\]
where $\mathcal{T}(u)$ is the only solution of 
\begin{equation*}
\begin{cases}
\partial_t \mathcal{T}(u) - \nu \Delta \mathcal{T}(u) =  \tilde{f} \text{ in } B;\\
\mathcal{T}(u) = Mu \text{ a.e. in } A ;\\
u(0) = u^*(0);
\end{cases}
\end{equation*}
where $Mu$ is defined by $$Mu(t,x) = \inf_{\xi \in K} k(t,x) + u(t,x + \xi).$$
For the case of a Lipschitz boundary we refer to \citep{calvo2017parabolic} and to \citep{gianazza1996abstract} for the case in which $B$ is increasing. Moreover the application $\mathcal{T}$ is monotone. Next we claim using classical results on QVI (see \citep{bensoussan1984impulse}) that there exists $v_1, \tilde{v}_2 \in H$, respective solutions of the two following $QVI$:
\[
\begin{aligned}
&\begin{cases}
v_1 \leq Mv_1; \\
\forall w \in H, w \leq Mv_1 :\\
\int_0^T \int_{\mathbb{T}^d}(\partial_t v_1 - \nu \Delta v_1 - \tilde{f}, w - v_1)_{H^{-1}\times H^1} \geq 0;\\
v_1(0) = u^*(0).
\end{cases}\\
&\begin{cases}
\tilde{v}_2 \leq M\tilde{v}_2; \\
\forall w \in H, w \leq M\tilde{v}_2 :\\
\int_0^T \int_{\mathbb{T}^d}(\partial_t \tilde{v}_2 - \nu \Delta \tilde{v}_2 + \tilde{f} - 2(\partial_t u^* - \nu \Delta u^*), w - \tilde{v}_2)_{H^{-1}\times H^1} \geq 0;\\
\tilde{v}_2(0) = u^*(0).
\end{cases}
\end{aligned}
\]
We define $v_2 = 2u^* - \tilde{v_2}$. Because $\mathcal{T}$ is monotone, the set $\mathcal{J} := \{ v \in H, v_1 \leq v \leq v_2\}$ is invariant by $\mathcal{T}$. Moreover this set is non-empty because $v_1 \leq v_2$. Thus $\mathcal{T}$ has a fixed point $u \in \mathcal{J}$. Now let us remark that as $||f||_{L^{\infty}}$ goes to $0$, $v_1$ and $v_2$ converges to $u^*$. Thus because of the assumption satisfied by $u^*$, namely the last line of (\ref{hypapp}), if $||f||_{L^{\infty}}$ is small enough, then because $u = Mu$ on $A$, $u$ satisfies 
\[
\forall \xi \in K, V(\xi,t,x) (k(x,\xi) + u(t, x+ \xi) - u(t,x)) = 0 \text{ a.e. in } (0,T) \times \mathbb{T}^d.
\]
\end{proof}

\subsection{Parabolic PDE in general time dependent domains}
In this section, we present a result for the heat equation on a time dependent domain with Dirichlet boundary conditions. We detail why under a rather general assumption on the evolution of the domain, we obtain a unique solution of the problem. We are interested in the following PDE :
\begin{equation}\label{tdpb}
\begin{cases}
\partial_t u - \nu \Delta u  = f \text{ in } B;\\
u = \phi \text{ in } A;\\
u(0) = \phi(0) \text{ in } \mathbb{T}^d.
\end{cases}
\end{equation}
In this PDE, $B \subset (0,T) \times \mathbb{T}^d$ and $A = \left( (0,T) \times \mathbb{T}^d \right) \setminus B$. We assume that :
\[
\phi \in H^1((0,T), H^1(\mathbb{T}^d)).
\]
Actual results of uniqueness concern situations in which the evolution of $B$ in time is either monotone, see \citep{gianazza1996abstract}, or smooth, see \citep{calvo2017parabolic}. We here make the following assumption :
\begin{hyp}\label{lasthyp}
The set $B$ is such that $µ = \partial_t \mathbb{1}_B$ is a measure whose positive and negative parts are strictly separated, i.e. there exists an open set $\mathcal{O}$ with $C^1$ boundary such that $supp(µ^-) \subset \mathcal{O}$ and $supp(µ^+) \cap \mathcal{O} = \emptyset$.
\end{hyp}
We plan our study as follows : proposition \ref{pexistenceapp} states the existence of weak solution of this problem ; lemma \ref{lemregapp} states a local regularity result needed for the uniqueness and theorem \ref{uniqueapp} is our main result of uniqueness.
\begin{Prop}\label{pexistenceapp}
There exists a weak solution $u \in L^2((0,T), H^1(\mathbb{T}^d))\cap L^{\infty}((0,T), L^2(\mathbb{T}^d))$ of (\ref{tdpb}).
\end{Prop}
\begin{proof}
For any $\epsilon > 0$, there exists a unique $u_{\epsilon} \in H$ solution of 
\[
\begin{cases}
\partial_t u_{\epsilon} - \nu \Delta u_{\epsilon} + \frac{1}{\epsilon} \mathbb{1}_{A}(u_{\epsilon} - \phi) = f \text{ in } (0,T) \times \mathbb{T}^d;\\
u_{\epsilon}(0) = \phi(0).
\end{cases}
\]
Multiplying this equation by $(u_{\epsilon} - \phi)$ and integrating in space we deduce that :
\[
\begin{aligned}
\frac{d}{dt}\frac{1}{2}\int_{\mathbb{T}^d} u_{\epsilon}^2(t) &+ \nu \int_{\mathbb{T}^d}|\nabla u_{\epsilon}|^2(t) + \frac{1}{\epsilon}\int_{A(t)}(u_{\epsilon} - \phi)^2(t) = - \frac{1}{2}\int_{\mathbb{T}^d} \phi(0)^2 \\
&+ \int_{\mathbb{T}^d} f (u_{\epsilon} - \phi)(t) + \nu \int_{\mathbb{T}^d} \nabla \phi \cdot \nabla u_{\epsilon}(t) + \int_{\mathbb{T}^d}\partial_t u_{\epsilon}\phi(t).
\end{aligned}
\]
Integrating between $0$ and $t$, we obtain :
\[
\begin{aligned}
&\frac{1}{2}\int_{\mathbb{T}^d}u_{\epsilon}^2(t) + \nu \int_0^t \int_{\mathbb{T}^d} |\nabla u_{\epsilon}|^2 + \frac{1}{\epsilon}\int_{\cup_{s \leq t} A(s)} (u_{\epsilon} - \phi)^2\\
& \leq \int_0^t \int_{\mathbb{T}^d} f (u_{\epsilon} - \phi)(t) + \nu \int_0^t\int_{\mathbb{T}^d} \nabla \phi \cdot \nabla u_{\epsilon} - \int_0^t \int_{\mathbb{T}^d} u_{\epsilon}\partial_t \phi + \int_{\mathbb{T}^d} \phi(t) u_{\epsilon}(t).
\end{aligned}
\]
Thus, the sequence $(u_{\epsilon})_{\epsilon}$ is bounded in $L^2((0,T), H^1(\mathbb{T}^d))\cap L^{\infty}((0,T), L^2(\mathbb{T}^d))$, hence it has a weak limit $u$ which satisfies all the required properties.
\end{proof}
We now pass to the proof of the lemma on local regularity.
\begin{Lemma}\label{lemregapp}
Let us define $u_{\epsilon}$ as in the proof of proposition \ref{pexistenceapp}. Assume that $u_{\epsilon} \in H$ is such that on an open set $(t_1,t_2) \times \Omega \subset (0,T) \times \mathbb{T}^d$ :
\[
\partial_t u_{\epsilon} - \nu \Delta u_{\epsilon} + \frac{1}{\epsilon} \lambda (u_{\epsilon} - \phi) = f;
\]
where $\lambda \in L^{\infty}((0,T) \times \mathbb{T}^d)$ is such that $ 0 \leq \lambda \leq 1$ and $\partial_t \lambda \leq 0$.

Then, for every open set $\Omega' $ such that $\overline{\Omega'}\subset \Omega$, there exists $C$ independent of $\epsilon$ such that 
\[
||\partial_t u_{\epsilon}||^2_{L^2} \leq C
\]
\end{Lemma}
\begin{proof}
We denote by $\zeta$ a positive $C^{\infty}$ function with compact support in $(t_1, t_2]\times \Omega$ such that $\zeta = 1$ on $\Omega'$. We multiply the equation satisfied by $u_{\epsilon}$ by $\zeta\partial_{t}(u_{\epsilon} - \phi)$ and we integrate. We obtain :
\[
\begin{aligned}
\int_{\Omega} \zeta \partial_t u_{\epsilon}^2 &+ \nu \frac{d}{dt} \frac{1}{2}\int_{\Omega}\zeta |\nabla u_{\epsilon}|^2 + \frac{d}{dt}\frac{1}{2\epsilon} \int_{\Omega}\lambda \zeta(u_{\epsilon} - \phi)^2 = \int_{\Omega}f\zeta\partial_t(u_{\epsilon} - \phi) + \int_{\Omega}\zeta\partial_t u_{\epsilon} \partial_t \phi \\
& - \nu \int_{\Omega}\zeta \Delta u_{\epsilon}\partial_t \phi + \frac{1}{2\epsilon} \int_{\Omega}\lambda (u_{\epsilon} - \phi)^2 \partial_t \zeta + \frac{1}{2\epsilon}\int_{\Omega}\zeta (u_{\epsilon} - \phi)^2 \partial_t \lambda + \nu \frac{1}{2} \int_{\Omega}|\nabla u_{\epsilon}|^2\partial_t \zeta.
\end{aligned}
\]
Integrating between $t_1$ and $s$ such that $t_1 \leq s \leq t_2$, we deduce that (because $\partial_t \lambda \leq 0$ and $\zeta(t_1) = 0$) :
\[
\begin{aligned}
\int_{t_1}^{s}\int_{\Omega'} \partial_t u_{\epsilon}^2 &+ \nu \frac{1}{2}\int_{\Omega'} |\nabla u_{\epsilon}|^2(s) + \frac{1}{2\epsilon} \int_{\Omega'}\lambda (u_{\epsilon} - \phi)^2(s) \leq \int_{t_1}^{s}\int_{\Omega}f\zeta\partial_t(u_{\epsilon} - \phi) + \int_{t_1}^{s}\int_{\Omega}\zeta\partial_t u_{\epsilon} \partial_t \phi \\
& - \nu \int_{t_1}^{s}\int_{\Omega}\zeta \Delta u_{\epsilon}\partial_t \phi + \frac{1}{2\epsilon} \int_{t_1}^{s}\int_{\Omega}\lambda (u_{\epsilon} - \phi)^2 \partial_t \zeta + \nu \frac{1}{2} \int_{t_1}^{s}\int_{\Omega}|\nabla u_{\epsilon}|^2\partial_t \zeta.
\end{aligned}
\]
Thus the required estimate follows.
\end{proof}
We can now pass to the main result of this section.
\begin{Theorem}\label{uniqueapp}
Assume hypothesis \ref{lasthyp} holds, then there exists a unique weak solution $u\in L^2((0,T), H^1(\mathbb{T}^d))$ of (\ref{tdpb}) such that for any set $\Omega \subset \mathcal{O}$ (where $\mathcal{O}$ is given by hypothesis \ref{lasthyp}), $\partial_t u \in L^2(\Omega)$.
\end{Theorem}
\begin{proof}
A weak solution $u$ of (\ref{tdpb}) exists by proposition \ref{pexistenceapp}. By lemma \ref{lemregapp}, it has the required time regularity. Let us assume that there exists two such weak solutions $u_1$ and $u_2$. We then define for $i = 1, 2$ $v_i \in L^2((0,T), H^1(\mathbb{T}^d))$, a weak solution of respectively 
\[
\begin{cases}
-\partial_t v_i - \nu \Delta v_i  = u_i \text{ in } B;\\
v_i = 0 \text{ in } A;\\
v_i(T) = 0 \text{ in } \mathbb{T}^d.
\end{cases}
\]
Using lemma \ref{lemregapp} (in the other time direction), we can assume that there exists an open set $\Omega_1 \subset (0,T) \times \mathbb{T}^d$ with smooth boundary such that 
\[
\begin{cases}
u_1, u_2 \in H^1(\Omega_1);\\
v_1, v_2 \in H^1(\Omega_2);
\end{cases}
\]
where $\Omega_2 = \overset{\circ}{\Omega_1^c}$. The precise weak formulations satisfied by $u_1, u_2, v_1$ and $v_2$ can now be written. Namely, for any $w \in L^2((0,T), H^1(\mathbb{T}^d))$ such that $w \in H^1(\Omega^1)$ and $w = 0$ almost everywhere on $A$, the following holds for $i =1,2$ :
\[
\int_{\Omega_1} v_i\partial_t w - \int_{\Omega_2} w\partial_tv_i+ \nu \int_0^T \int_{\mathbb{T}^d}\nabla w \cdot \nabla v_i = \int_0^T\int_{\mathbb{T}^d} u_i w + \int_{\partial \Omega_1}wv_i\eta^1_t;
\]
where $\eta^1 = (\eta^1_t, \eta^1_x)$ is the exterior unit normal vector to $\Omega_1$. Conversely, for any $w \in L^2((0,T), H^1(\mathbb{T}^d))$ such that $w \in H^1(\Omega^2)$ and $w = 0$ almost everywhere on $A$, the following holds for $i =1,2$
\[
-\int_{\Omega_2} u_i\partial_t w + \int_{\Omega_1} w\partial_tu_i+ \nu \int_0^T \int_{\mathbb{T}^d}\nabla w \cdot \nabla u_i = \int_0^T\int_{\mathbb{T}^d} fw + \int_{\partial \Omega_2}wu_i\eta^2_t;
\]
where $\eta^2 = (\eta^2_t, \eta^2_x)$ is the exterior unit normal vector to $\Omega_2$.Thus we deduce, using this relation for $w = u_1 - u_2$ on $v_1$ and $v_2$ ; and for $w = v_1 - v_2$ on $u_1$ and $u_2$, that
\[
\int_0^T\int_{\mathbb{T}^d} (u_1 - u_2)^2 = 0.
\]
Thus there exists a unique weak solution which has the time regularity required.
\end{proof}

\end{document}